\documentclass[12pt]{amsart}
\usepackage{amsmath, amssymb, latexsym}
\usepackage{mathrsfs}
\usepackage{enumerate}
\usepackage{color}
\usepackage[colorlinks=true, pdfstartview=FitV, linkcolor=blue,%
citecolor=blue, urlcolor=blue]{hyperref}
\usepackage[all, knot]{xy}
\xyoption{arc}

\newtheorem{Thm}{Theorem}[section]
\newtheorem{Prop}[Thm]{Proposition}
\newtheorem{Cor}[Thm]{Corollary}
\newtheorem{Lem}[Thm]{Lemma}
\newtheorem{Def}[Thm]{Definition}

\theoremstyle{definition}

\newtheorem{Rem}[Thm]{Remark}

\newenvironment{red}
{\relax\color{red}}
{\hspace*{.3ex}\relax}
\newcommand{\ber}{\begin{red}}
\newcommand{\er}{\end{red}}

\numberwithin{equation}{section}

\newcommand{\nc}{\newcommand}

\renewcommand{\to}[1][]{\xrightarrow{#1}{}}

\newcommand{\Z}{\mathbf{Z}}
\newcommand{\Q}{\mathbf{Q}}

\newcommand{\A}{\mathbf{A}}

\newcommand{\g}{\mathfrak{g}}
\newcommand{\gsl}{\mathfrak{sl}}

\newcommand{\wt}{{\rm wt}}

\nc{\on}{\operatorname}

\newcommand{\Hom}{\on{Hom}}

\newcommand{\End} { {\rm End}}
\newcommand{\id} { {\rm id}}

\nc{\cor}{\mathbf{k}}
\nc{\cori}{\mathbf{k}^I}
\nc{\KLR}{Khovanov-Lauda-Rouquier algebra}
\nc{\KLRs}{Khovanov-Lauda-Rouquier algebras}
\nc{\seteq}{\mathbin{:=}}
\newcommand{\soplus}{\mathop{\mbox{\normalsize$\bigoplus$}}\limits}
\nc{\cl}{\colon}
\nc{\set}[2]{\left\{#1\mid #2\right\}}
\nc{\Id}{\operatorname{Id}}
\nc{\Ker}{\on{Ker}}
\nc{\Coker}{\on{Coker}}
\nc{\coh}{\mathrm{coh}}
\nc{\Mod}{\on{Mod}}
\nc{\Modc}{\on{Mod_\coh}}
\nc{\Proj}{\on{Proj}}
\nc{\Rep}{\on{Rep}}
\newcommand{\isoto}[1][]{\mathop{\xrightarrow[#1]%
{{\raisebox{-.6ex}[0ex][-.6ex]{$\mspace{2mu}\sim\mspace{2mu}$}}}}}
\newcommand{\isofrom}[1][]{\xleftarrow[#1]{\sim}}
\nc{\To}[1][\quad]{\to[\;#1\;]}
\nc{\hs}{\hspace*}
\nc{\vs}{\vspace*}
\nc{\bF}{\overline{F}}
\nc{\epi}{\twoheadrightarrow}
\nc{\mono}{\rightarrowtail}
\nc{\be}{\begin{enumerate}}
\nc{\ee}{\end{enumerate}}
\nc{\ba}{\begin{array}}
\nc{\ea}{\end{array}}
\nc{\eq}{\begin{eqnarray}}
\nc{\eneq}{\end{eqnarray}}
\nc{\eqn}{\begin{eqnarray*}}
\nc{\eneqn}{\end{eqnarray*}}
\nc{\ran}{\rangle}
\nc{\lan}{\langle}
\nc{\bl}{\bigl(}
\nc{\br}{\bigl)}
\nc{\bnum}{\be[{\rm(i)}]}
\nc{\enum}{\ee}
\nc{\bna}{\be[{\rm(a)}]}
\nc{\Proof}{\begin{proof}}
\nc{\QED}{\end{proof}}
\newcommand{\scbul}{{\,\raise1pt\hbox{$\scriptscriptstyle\bullet$}\,}}
\nc{\tens}{\mathop\otimes}
\nc{\E}[1][i]{\mathsf{E}_{#1}}
\nc{\F}[1][i]{\mathsf{F}_{#1}}
\nc{\x}[1][{1}]{a^\Lambda(x_{#1})}
\nc{\Supp}{\mathrm{Supp}}
\nc{\vphi}{\varphi}
\nc{\la}{\lambda}

\newlength{\my}
\begin{document}

\title[Categorification via Khovanov-Lauda-Rouquier Algebras]
{Categorification of Highest Weight Modules  \\ via
Khovanov-Lauda-Rouquier Algebras}
\author[Seok-Jin Kang]{Seok-Jin Kang $^{1}$}
\thanks{$^1$ This work was supported by KRF Grant \# 2007-341-C00001 and NRF Grant \# 2010-0010753.}
\address{Department of Mathematical Sciences and Research Institute of Mathematics,
Seoul National University, 599 Gwanak-ro, Gwanak-gu, Seoul 151-747,
Korea} \email{sjkang@snu.ac.kr}

\author[Masaki Kashiwara]{Masaki Kashiwara $^{2}$}
\thanks{$^2$ This work was partially supported by Grant-in-Aid for
Scientific Research (B) 22340005, Japan Society for the Promotion of Science.}
\address{Research Institute for Mathematical Sciences, Kyoto University, Kyoto 606-8502, Japan,
and Department of Mathematical Sciences, Seoul National University,
599 Gwanak-ro, Seoul 151-747, Korea}
\email{masaki@kurims.kyoto-u.ac.jp}



\begin{abstract}
In this paper, we prove Khovanov-Lauda's cyclotomic
categorification conjecture for all symmetrizable Kac-Moody
algebras. Let $U_q(\g)$ be the quantum group
associated with a symmetrizable Cartan datum and let $V(\Lambda)$ be
the irreducible highest weight $U_q(\g)$-module with a dominant
integral highest weight $\Lambda$. We prove that the cyclotomic
Khovanov-Lauda-Rouquier algebra $R^{\Lambda}$ gives a
categorification of $V(\Lambda)$.
\end{abstract}

\maketitle


\section{Introduction}

The {\it Khovanov-Lauda-Rouquier algebras}, a vast generalization of
affine Hecke algebras of type $A$, were introduced independently by
Khovanov and Lauda (\cite{KL09, KL11})
and Rouquier (\cite{R08}) to provide a {\em categorification} of
quantum groups. Let $U_q(\g)$ be the quantum
group associated with a symmetrizable Cartan datum and let
$R=\bigoplus_{\beta \in Q^{+}} R(\beta)$ be the corresponding
Khovanov-Lauda-Rouquier algebra. Then it was shown in \cite{KL09,
KL11, R08} that there exists an algebra isomorphism
$$U_{\A}^{-}(\g) \simeq [\Proj(R)]= \bigoplus_{\beta \in
Q^{+}}[\Proj(R(\beta))],$$ where $U_{\A}^-(\g)$ is the integral form
of the negative half $U_q^-(\g)$ of $U_{q}(\g)$ with $\A = \Z[q,
q^{-1}]$, and $[\Proj(R)]$ is the Grothendieck group of the additive
category of finitely generated graded projective $R$-modules.
Moreover, when the generalized Cartan matrix is a symmetric matrix,
Varagnolo and Vasserot (\cite{VV09}) and independently Rouquier (
\cite{R11}) proved that Kashiwara's {\em lower global basis} or
Lusztig's {\it canonical basis} corresponds to the isomorphism
classes of indecomposable projective $R$-modules under this
isomorphism.

For each dominant integral weight $\Lambda \in P^{+}$, the algebra
$R$ has a special quotient $R^{\Lambda}=\bigoplus_{\beta \in Q^{+}}
R^{\Lambda}(\beta)$ which is called the {\it cyclotomic
Khovanov-Lauda-Rouquier algebra}. In \cite{KL09}, Khovanov and Lauda
conjectured that $[\Proj(R^{\Lambda})]$ has a $U_\A(\g)$-module
structure  and that there exists a $U_\A(\g)$-module isomorphism
$$V_\A(\Lambda) \simeq [\Proj(R^{\Lambda})]
= \bigoplus_{\beta \in Q^{+}} [\Proj(R^{\Lambda}(\beta))],$$
where $V_\A(\Lambda)$ denotes the $U_\A(\g)$-module generated by the
highest weight vector $v_{\Lambda}$. It is called the {\it
cyclotomic categorification conjecture}. In \cite{BS08}, Brundan and
Stroppel proved a special case of this conjecture in finite type
$A$.
 In \cite{BK09}, Brundan and Kleshchev proved
this conjecture for type $A_\infty$ and $A_n^{(1)}$ using the
isomorphism between $R^{\Lambda}$ and the cyclotomic Hecke algebra
$H^{\Lambda}$ which was constructed in \cite{BK08}.
The $\mathfrak{sl}_{2}$-categorification theory
developed in \cite{CR08, R08} also
played an important role in their proof. In \cite{LV09}, the crystal
version of this conjecture was proved for all symmetrizable
Kac-Moody algebras. That is, in \cite{LV09}, Lauda and Vazirani
investigated the crystal structure on the set of isomorphism classes
of irreducible graded modules over $R$ and $R^\Lambda$, and showed
that these crystals are isomorphic to the crystals $B(\infty)$ and
$B(\Lambda)$, respectively.

In this paper, we prove Khovanov-Lauda's cyclotomic categorification
conjecture for {\it all} symmetrizable Kac-Moody algebras. For
$\beta \in Q^{+}$, let $\Mod(R^{\Lambda}(\beta))$ be the abelian
category of $\Z$-graded $R^{\Lambda}(\beta)$-modules. For each
$i \in I$, let us consider the restriction functor and the induction
functor:
\begin{equation*}
\begin{aligned}
& E_{i}^{\Lambda}\cl \Mod(R^{\Lambda}(\beta+\alpha_i)) \longrightarrow \Mod(R^{\Lambda}(\beta)), \\
& F_{i}^{\Lambda}\cl \Mod(R^{\Lambda}(\beta)) \longrightarrow
\Mod(R^{\Lambda}(\beta+ \alpha_i))
\end{aligned}
\end{equation*}
defined by
\begin{equation*}
\begin{aligned}
& E_{i}^{\Lambda}(N) = e(\beta, i) N = e(\beta, i)R^{\Lambda}(\beta+
\alpha_i) \otimes_{R^{\Lambda}(\beta+\alpha_i)} N, \\
& F_{i}^{\Lambda}(M) = R^{\Lambda}(\beta+\alpha_i) e(\beta, i)
\otimes_{R^{\Lambda}(\beta)} M,
\end{aligned}
\end{equation*}
where $M \in \Mod(R^{\Lambda}(\beta))$, $N \in
\Mod(R^{\Lambda}(\beta+\alpha_i))$.

Our first main result is that
$R^{\Lambda}(\beta+\alpha_i)e(\beta, i)$ is a projective right
$R^\Lambda(\beta)$-module and $e(\beta,
i)R^{\Lambda}(\beta+\alpha_i)$ is a projective left
$R^\Lambda(\beta)$-module (Theorem \ref{th:proj}).
Hence the functors $E_{i}^{\Lambda}$ and $F_{i}^{\Lambda}$ are exact and
send projectives to projectives.

\medskip
Another main result of this paper can be summarized as follows
(Theorem \ref{thm:M}): let $\lambda = \Lambda - \beta$.
\be[{\rm(1)}]
\item If $\langle h_{i}, \lambda \rangle \ge 0$, there exists a
natural isomorphism of endofunctors on $\Mod(R^{\Lambda}(\beta))$:
$$q_{i}^{-2} F_{i}^{\Lambda} E_{i}^{\Lambda} \oplus
\bigoplus_{k=0}^{\langle h_{i}, \lambda \rangle -1} q_{i}^{2k} \Id
\overset{\sim} \longrightarrow E_{i}^{\Lambda} F_{i}^{\Lambda}.$$

\item If $\langle h_{i}, \lambda \rangle \le 0$, there exists a
natural isomorphism of endofunctors on $\Mod(R^{\Lambda}(\beta))$:
$$q_{i}^{-2} F_{i}^{\Lambda} E_{i}^{\Lambda} \overset{\sim}
\longrightarrow E_{i}^{\Lambda} F_{i}^{\Lambda} \oplus
\bigoplus_{k=0}^{-\langle h_{i}, \lambda \rangle -1} q_{i}^{-2k-2}
\Id.$$ \ee Here, $q_i\seteq q^{(\alpha_i,\alpha_i)/2}$ denotes the
grade-shift functor defined in \eqref{eq:shift}.  This is one of the
axioms of the categorification of $U_q(\g)$ due to Chuang-Rouquier
\cite{CR08} and Rouquier \cite{R08}.

We write $[\Rep(R^{\Lambda})]$
for the Grothendieck group of the abelian category
$\Rep(R^{\Lambda})$ of $R^\Lambda$-modules that are
finite-dimensional over the base field. It follows that the functors
$E_{i}^{\Lambda}$, $F_{i}^{\Lambda}$ $(i \in I)$ satisfy the mixed
relations (Lemma \ref{lem:mixed}), and hence by \cite[Proposition
B.1]{KMPY96}, the Grothendieck groups $[\Proj(R^{\Lambda})]$ and
$[\Rep(R^{\Lambda})]$ become integrable $U_q(\g)$-modules.
Therefore, we obtain the categorification of the irreducible highest
weight module $V(\Lambda)$ (Theorem \ref{thm:N}):
$$[\Proj(R^{\Lambda})] \simeq V_{\A}(\Lambda) \ \ \text{and} \ \
[\Rep(R^{\Lambda})] \simeq V_{\A}(\Lambda)^{\vee},$$ where
$V_{\A}(\Lambda)^{\vee}$ is the dual of $V_{\A}(\Lambda)$ with
respect to a non-degenerate symmetric bilinear form on $V(\Lambda)$.
In other words, we obtain an integrable $2$-representation of the
$2$-Kac-Moody algebra in the sense of Rouquier \cite[Definition
5.1]{R08}.

\bigskip
One of the key ingredients of the proof of these results is a
categorification of the equality
\eq
&&[e_i,P]=\dfrac{K_i^{-1}e_i'(P)-K_ie_i''(P)} {q_i^{-1}-q_i}
\quad\text{for $P\in U_q^-(\g)$} \label{eq:ei} \eneq
used in \cite{Kash91} in the course of constructing the theory of crystal
bases. Here $e'_i$ and $e_i''$ are endomorphisms of $U_q^{-}(\g)$.
Hence, for the highest weight vector $v_{\Lambda}$ of $V(\Lambda)$,
we have
$$e_i(Pv_\Lambda)=(q^{-1}-q)^{-1}
\Bigl(q^{(\alpha_i |\Lambda+\wt(P))}e_i'(P)v_\Lambda -q^{-(\alpha_i
|\Lambda+\wt(P))}e_i''(P)v_\Lambda\Bigr).$$ By the categorification,
the operator $e_i$ corresponds (after taking the adjoints) to the
functor $F_i^\Lambda$, while the operators $(q^{-1}-q)^{-1}e'_i$ and
$(q^{-1}-q)^{-1}e_i''$ correspond to the functors \eqn
F_i(M)&=&M\circ R(\alpha_i)=R(n+1)e(n,i)\otimes_{R(n)}M\quad\text{and}\\
\bF_i(M)&=&R(\alpha_i)\circ M=R(n+1)e(i,n)\otimes_{R(n)}M,\text{\ respectively.}
\eneqn
Here the convolution functor $\scbul\circ\scbul\cl
\Mod(R(m))\times\Mod(R(n))\to \Mod(R(m+n))$ is defined by
$M\circ N=R(m+n)\otimes_{R(m)\otimes R(n)}(M\otimes N)$.
Then the categorification of the identity \eqref{eq:ei}
can be interpreted as an exact sequence (see Theorem~\ref{th:main})
\eq
0\to \bF_iM\to F_i M\to F_i^\Lambda M\to 0
\quad\text{for any $M\in\Mod(R^\Lambda(n))$.}\label{exact:main}
\eneq
Our main results are consequences of this exact sequence.

\medskip
In \cite{Web10}, Webster gave a proof of
Khovanov-Lauda's cyclotomic categorification
conjecture by a completely different
method, which is beyond the authors' comprehension.

\bigskip
This paper is organized as follows. In Section 2 and Section 3, we
recall basic properties of quantum groups, integrable highest weight
modules and the Khovanov-Lauda-Rouquier algebra $R$. In Section 4,
we investigate the structure of cyclotomic Khovanov-Lauda-Rouquier
algebra $R^{\Lambda}$ and prove the exact sequence
\eqref{exact:main}, and then show that the functors
$E_{i}^{\Lambda}$, $F_{i}^{\Lambda}$ $(i\in I)$ are exact and send
projectives to projectives. Section 5 is devoted to the
$\mathfrak{sl}_{2}$-categorification theory. In Section 6, we finish
the proof of Khovanov-Lauda's cyclotomic categorification
conjecture.

\vskip 5mm

{\it Acknowledgements.} The first author would like to express his
sincere gratitude to Research Institute for Mathematical Sciences,
Kyoto University for their hospitality during his visit in January,
2011.
We would also like to thank Se-jin Oh for many helpful discussions.
Special thanks should be given to the referees for their valuable suggestions
which have greatly improved the exposition of the original manuscript.
\vskip 3em


\section{Quantum groups and highest weight modules} \label{sec:qgroup}

Let $I$ be a finite index set. An integral square matrix
$A=(a_{ij})_{i,j \in I}$ is called a {\em symmetrizable generalized
Cartan matrix} if it satisfies (i) $a_{ii} = 2$ $(i \in I)$, (ii)
$a_{ij} \le 0$ $(i \neq j)$, (iii) $a_{ij}=0$ if $a_{ji}=0$ $(i,j \in I)$,
(iv) there is a diagonal matrix
$D=\text{diag} (d_i \in \Z_{> 0} \mid i \in I)$ such that $DA$ is
symmetric.

A \emph{Cartan datum} $(A,P, \Pi,P^{\vee},\Pi^{\vee})$ consists of
\begin{enumerate}
\item[(1)] a symmetrizable generalized Cartan matrix $A$,
\item[(2)] a free abelian group $P$ of finite rank, called the \emph{weight lattice},
\item[(3)] $\Pi= \{ \alpha_i \in P \mid \ i \in I \}$, called
the set of \emph{simple roots},
\item[(4)] $P^{\vee}\seteq\Hom(P, \Z)$, called the \emph{dual weight lattice},
\item[(5)] $\Pi^{\vee}= \{ h_i \ | \ i \in I  \}\subset P^{\vee}$, called
the set of \emph{simple coroots},
\end{enumerate}
satisfying the following properties:
\begin{enumerate}
\item[(i)] $\langle h_i,\alpha_j \rangle = a_{ij}$ for all $i,j \in I$,
\item[(ii)] $\Pi$ is linearly independent,
\item[(iii)] for each $i \in I$, there exists $\Lambda_i \in P$ such that
           $\langle h_j ,\Lambda_i \rangle =\delta_{ij}$ for all $j \in I$.
\end{enumerate}

The $\Lambda_i$ are called the {\em fundamental weights}. We denote by
$$P^{+} \seteq \set{ \lambda \in P}%
{\text{$\langle h_i, \lambda \rangle \in\Z_{\ge 0}$ for all $i \in I$}}$$
the set of \emph{dominant integral weights}.
The free abelian group $Q\seteq\soplus_{i \in I} \Z
\alpha_i$ is called the \emph{root lattice}. Set $Q^{+}= \sum_{i \in
I} \Z_{\ge 0} \alpha_i$. For $\alpha = \sum k_i \alpha_i \in Q^{+}$,
we define the {\it height} of $|\alpha|$ to be $|\alpha|=\sum k_i$.
Let $\mathfrak{h} = \Q \otimes_\Z P^{\vee}$. Since $A$ is
symmetrizable, there is a symmetric bilinear form $(\quad|\quad)$ on
$\mathfrak{h}^*$ satisfying
$$ (\alpha_i | \alpha_j) =d_i a_{ij} \quad (i,j \in I)
\quad\text{and $\lan h_i,\lambda\ran=
\dfrac{2(\alpha_i|\lambda)}{(\alpha_i|\alpha_i)}$ for any $\lambda\in\mathfrak{h}^*$ and $i \in I$}.$$

Let $q$ be an indeterminate and set $q_i = q^{\frac{(\alpha_i|
\alpha_i)}{2}}$. Note that $(\alpha_i | \alpha_i)={2}d_i\in 2\Z_{>0} $.
For $m,n \in \Z_{\ge 0}$, we define
\begin{equation*}
 \begin{aligned}
 &[n]_i =\frac{ q^n_{i} - q^{-n}_{i} }{ q_{i} - q^{-1}_{i} },\quad
 &[n]_i! = \prod^{n}_{k=1} [k]_i\; ,\quad
 &\left[\begin{matrix}m \\ n\\ \end{matrix} \right]_i=  \frac{ [m]_i! }{[m-n]_i! [n]_i! }\;.
 \end{aligned}
\end{equation*}

\begin{Def} \label{def:qgroup}
The {\em quantum group} $U_q(\g)$ associated with a Cartan datum
$(A,P,\Pi,P^{\vee}, \Pi^{\vee})$ is the associative algebra over
$\Q(q)$ with $1$ generated by $e_i,f_i$ $(i \in I)$ and $q^{h}$ $(h
\in P^{\vee})$ satisfying the  following relations:
\begin{equation}
\begin{aligned}
& q^0=1,\  q^{h} q^{h'}=q^{h+h'} \ \ \text{for} \ h,h' \in P^{\vee},\\
& q^{h}e_i q^{-h}= q^{\langle h, \alpha_i \rangle} e_i, \ \
          \ q^{h}f_i q^{-h} = q^{-\langle h, \alpha_i \rangle} f_i \ \ \text{for} \ h \in P^{\vee}, i \in
          I, \\
& e_if_j - f_je_i =  \delta_{ij} \dfrac{K_i -K^{-1}_i}{q_i- q^{-1}_i
}, \ \ \mbox{ where } K_i=q_i^{ h_i}, \\
& \sum^{1-a_{ij}}_{r=0} \left[\begin{matrix}1-a_{ij}
\\ r\\ \end{matrix} \right]_i e^{1-a_{ij}-r}_i
         e_j e^{r}_i =0 \quad \text{ if } i \ne j, \\
& \sum^{1-a_{ij}}_{r=0} \left[\begin{matrix}1-a_{ij}
\\ r\\ \end{matrix} \right]_i f^{1-a_{ij}-r}_if_j
        f^{r}_i=0 \quad \text{ if } i \ne j.
\end{aligned}
\end{equation}
\end{Def}

Let $U_q^{+}(\g)$ (resp.\ $U_q^{-}(\g)$) be the subalgebra of
$U_q(\g)$ generated by $e_i$'s (resp.\ $f_i$'s), and let $U^0_q(\g)$
be the subalgebra of $U_q(\g)$ generated by $q^{h}$ $(h \in
P^{\vee})$. Then we have the \emph{triangular decomposition}
$$ U_q(\g) \cong U^{-}_q(\g) \otimes U^{0}_q(\g) \otimes U^{+}_q(\g),$$
and the {\em weight space decomposition}
$$U_q(\g) = \bigoplus_{\alpha \in Q} U_q(\g)_{\alpha},$$
where $U_q(\g)_{\alpha}\seteq\set{ x \in U_q(\g)}{q^{h}x q^{-h}
=q^{\langle h, \alpha \rangle}x \text{ for any } h \in P^{\vee}}$.

Let $\A= \Z[q, q^{-1}]$ and set
$$e_i^{(n)} = e_i^n / [n]_i!, \quad f_i^{(n)} =
f_i^n / [n]_i! \ \ (n \in \Z_{\ge 0}).$$
We define the $\A$-form
$U_{\A}(\g)$ to be the $\A$-subalgebra of $U_q(\g)$ generated by
$e_i^{(n)}$, $f_i^{(n)}$ $(i \in I, n \in \Z_{\ge 0})$, $q^h$ ($h\in P^\vee$).
Let $U_{\A}^{+}(\g)$ (resp.\ $U_{\A}^{-}(\g)$) be the
$\A$-subalgebra of $U_q(\g)$ generated by $e_i^{(n)}$ (resp.\
$f_i^{(n)}$) for $i\in I$, $n \in \Z_{\ge 0}$.

\begin{Def} \hfill
\bna
\item A $U_q(\g)$-module $M$ is called a {\em weight module} if
it has a weight space decomposition
$$M=\bigoplus_{\mu \in P} M_{\mu}, \ \text{where} \ M_{\mu}\seteq\set{ v \in M}
{q^h v = q^{\mu(h)} v \ \ \text{for all} \ h \in P^{\vee} }.$$

\item
A weight module $M$ is called {\em integrable} if the actions of
$e_i$ and $f_i$ on $M$ are locally nilpotent for any $i\in I$; i.e.,
for each $s\in M$ there exists a positive integer $m$ such that
$e_i^ms=f_i^ms=0$ for any $i\in I$.

\item A weight module $V$ is called a {\em highest weight
module with highest weight $\Lambda \in P$} if there exists a
non-zero vector $v_{\Lambda} \in V$ such that

\bnum
\item $e_i v_{\Lambda} = 0$ for all $i \in I$,
\item $q^h v_{\Lambda} = q^{\langle h, \Lambda \rangle}
v_{\Lambda}$ for all $h \in P^{\vee}$,
\item $V= U_q(\g)v_{\Lambda}$.
\end{enumerate}

\end{enumerate}
\end{Def}
For each $\Lambda \in P$, there exists a unique irreducible
highest weight module $V(\Lambda)$ with highest weight $\Lambda$.

\begin{Prop} [\cite{HK02, Lus93}] \label{prop:hw} Let $\Lambda \in P^{+}$.
\bna
\item If\/ $V$ is an integrable highest weight module with
highest weight $\Lambda$, then $V$ is
isomorphic to $V(\Lambda)$.
\item The highest weight vector
$v_{\Lambda}$ in $V(\Lambda)$ satisfies the following relations:
\begin{equation}\label{eq:hwvector}
f_i^{\langle h_i, \Lambda \rangle +1} v_{\Lambda} =0 \ \ \text{for
all} \ i \in I.
\end{equation}
\end{enumerate}
\end{Prop}

Consider the anti-involution $\phi$ on $U_q(\g)$ defined by
\begin{equation*}
q^h \mapsto q^{h}, \quad e_i \mapsto f_i,\quad f_i\mapsto e_i.
\end{equation*}
By standard arguments, one can show that there exists a unique
non-degenerate symmetric bilinear form $( \ , \ )$ on $V(\Lambda)$
with $\Lambda \in P^{+}$ satisfying
\begin{equation}
\begin{aligned}
& \text{$(v_{\Lambda}, v_{\Lambda})=1$ and
$(au,v)=(u, \phi(a)v)$
for all $a\in U_q(\g)$ and $u, v \in V(\Lambda)$.}
\end{aligned}
\end{equation}

We define the {\it $\A$-form} $V_{\A}(\Lambda)$ of $V(\Lambda)$ to
be $$V_{\A}(\Lambda) = U_{\A}(\g) v_{\Lambda}.$$ The {\it dual} of
$V_{\A}(\Lambda)$ is defined to be $$ V_{\A}(\Lambda)^{\vee} =
\set{ v\in V(\Lambda)}%
{\text{$(u, v ) \in \A$ for all $u\in V_{\A}(\Lambda)$}}.$$

We have
$V_\A(\Lambda)^\vee_\lambda\simeq\Hom_{\A}(V_\A(\Lambda)_\lambda,\A)$
for any $\lambda\in P$.

\vskip 3em


\section{The Khovanov-Lauda-Rouquier algebra} \label{sec:R}

Let $(A, P, \Pi, P^{\vee}, \Pi^{\vee})$ be a Cartan datum. In this
section, we recall the construction of \KLR\ $R$
associated with $(A, P, \Pi, P^{\vee}, \Pi^{\vee})$ and
investigate its properties.
We take  as a base ring a graded commutative ring
$\cor=\soplus\nolimits_{n\in\Z}\,\cor_n$ such that $\cor_n=0$ for any $n<0$.
Let us take a matrix $(Q_{ij})_{i,j\in I}$ in $\cor[u,v]$
such that $Q_{ij}(u,v)=Q_{ji}(v,u)$ and $Q_{ij}(u,v)$ has the form
\begin{equation} \label{eq:Q}
Q_{ij}(u,v) = \begin{cases}\hs{5ex} 0 \ \ & \text{if $i=j$,} \\
\sum\limits_{p,q\ge0}
t_{i,j;p,q} u^p v^q\quad& \text{if $i \neq j$,}
\end{cases}
\end{equation}
where $t_{i,j;p,q}\in \cor_{ -2(\alpha_i | \alpha_j)-(\alpha_i|\alpha_i) p - (\alpha_j|\alpha_j)q }$ and
 $t_{i,j;-a_{ij},0} \in \cor_0^\times$.
In particular, we have $t_{i,j;p,q}=0$ if
$(\alpha_i|\alpha_i) p + (\alpha_j|\alpha_j)q >-2(\alpha_i | \alpha_j)$.
Note that $t_{i,j;p,q} = t_{j,i;q,p}$.

We denote by
$S_{n} = \langle s_1, \ldots, s_{n-1} \rangle$ the symmetric group
on $n$ letters, where $s_i = (i, i+1)$ is the transposition.
Then $S_n$ acts on $I^n$.

\begin{Def}[\cite{KL09,{R08}}] \label{def:KLRalg}
The {\em \KLR}\ $R(n)$ of degree $n$
associated with a Cartan datum $(A, P, \Pi, P^{\vee}, \Pi^{\vee})$ and
$(Q_{ij})_{i,j\in I}$ is the associative algebra over $\cor$
generated by $e(\nu)$ $(\nu \in I^{n})$, $x_k$ $(1 \le k \le n)$,
$\tau_l$ $(1 \le l \le n-1)$ satisfying the following defining
relations:
\begin{equation} \label{eq:KLR}
\begin{aligned}
& e(\nu) e(\nu') = \delta_{\nu, \nu'} e(\nu), \ \
\sum_{\nu \in I^{n}}  e(\nu) = 1, \\
& x_{k} x_{l} = x_{l} x_{k}, \ \ x_{k} e(\nu) = e(\nu) x_{k}, \\
& \tau_{l} e(\nu) = e(s_{l}(\nu)) \tau_{l}, \ \ \tau_{k} \tau_{l} =
\tau_{l} \tau_{k} \ \ \text{if} \ |k-l|>1, \\
& \tau_{k}^2 e(\nu) = Q_{\nu_{k}, \nu_{k+1}} (x_{k}, x_{k+1})
e(\nu), \\
& (\tau_{k} x_{l} - x_{s_k(l)} \tau_{k}) e(\nu) = \begin{cases}
-e(\nu) \ \ & \text{if} \ l=k, \nu_{k} = \nu_{k+1}, \\
e(\nu) \ \ & \text{if} \ l=k+1, \nu_{k}=\nu_{k+1}, \\
0 \ \ & \text{otherwise},
\end{cases} \\[.5ex]
& (\tau_{k+1} \tau_{k} \tau_{k+1}-\tau_{k} \tau_{k+1} \tau_{k}) e(\nu)\\
&\hs{8ex} =\begin{cases} \dfrac{Q_{\nu_{k}, \nu_{k+1}}(x_{k},
x_{k+1}) - Q_{\nu_{k+2}, \nu_{k+1}}(x_{k+2}, x_{k+1})}
{x_{k} - x_{k+2}}e(\nu) \ \ & \text{if} \
\nu_{k} = \nu_{k+2}, \\
0 \ \ & \text{otherwise}.
\end{cases}
\end{aligned}
\end{equation}
\end{Def}

In particular, $R(0)\simeq\cor$,
and $R(1)$ is isomorphic to $\cori[x_1]$, where
$\cori=\oplus_{i\in I}\cor e(i)$ is the direct sum of
the copies
$\cor e(i)$ of the algebra $\cor$.

Note that $R(n)$ has an anti-involution $\psi$ that fixes the
generators $x_k$, $\tau_l$ and $e(\nu)$.

The $\Z$-grading on $R(n)$ is given by
\begin{equation} \label{eq:Z-grading}
\deg e(\nu) =0, \quad \deg\; x_{k} e(\nu) = (\alpha_{\nu_k}
| \alpha_{\nu_k}), \quad\deg\; \tau_{l} e(\nu) = -
(\alpha_{\nu_l} | \alpha_{\nu_{l+1}}).
\end{equation}

For $a,b,c \in \{1,\ldots,n\}$, we define the following
elements of $R(n)$ by
\begin{equation}
\begin{aligned}
& e_{a,b} = \sum_{\nu \in I^{n},\, \nu_{a}=\nu_{b}} e(\nu), \\
& Q_{a,b} = \sum_{\nu \in I^n} Q_{\nu_{a},\, \nu_{b}} (x_{a},
x_{b}) e(\nu), \\
& \overline{Q}_{a,b,c} = \sum_{\nu \in I^{n},\; \nu_{a} = \nu_{c}}
\dfrac{Q_{\nu_{a}, \nu_{b}}(x_{a}, x_{b}) - Q_{\nu_{a},
\nu_{b}}(x_{c}, x_{b})}{x_{a} - x_{c}} e(\nu) \quad\text{if $a\not=c$.}
\end{aligned}\label{def:Q}
\end{equation}
Then we have
\begin{equation}
\begin{aligned}
& Q_{a,b} = Q_{b,a}, \quad \tau_{a}^2 = Q_{a, a+1}, \\
& \tau_{a+1} \tau_{a} \tau_{a+1} = \tau_{a} \tau_{a+1} \tau_{a} +
\overline{Q}_{a, a+1, a+2}.
\end{aligned}\label{eq:tau3}
\end{equation}
We define the operators $\partial_{a,b}$ on $\soplus_{\nu\in
I^n}\cor[x_1, \ldots, x_n]e(\nu)$ by
\begin{equation}
\partial_{a,b} f = \dfrac{s_{a,b} f - f} {x_{a} - x_{b}}e_{a,b}, \quad
\partial_{a} = \partial_{a,a+1},
\end{equation}
where $s_{a,b} = (a,b)\in S_n$ is the transposition acting on
$\soplus_{\nu\in I^n} \cor[x_1, \ldots, x_n]e(\nu)$.

Thus we obtain
\begin{equation} \label{eq:partial}
\begin{aligned}
& \overline{Q}_{a,b,c} = - \partial_{a,c} Q_{a,b} = \partial_{a,c}
Q_{b,c}, \\
& \tau_{a} e_{b,c} = e_{s_a(b), s_a(c)} \tau_{a}, \\
& \tau_{a} f - (s_a f) \tau_{a} = f \tau_{a} - \tau_{a} (s_a f) =
(\partial_{a} f) e_{a,a+1}.
\end{aligned}
\end{equation}

For $n\in \Z_{\ge 0}$ and $\beta \in Q^{+}$ such that $|\beta|=n$, we set
$$I^{\beta} = \set{ \nu = (\nu_1, \ldots, \nu_n) \in I^n }%
{\alpha_{\nu_1} + \cdots + \alpha_{\nu_n} = \beta }.$$
We define
\begin{equation}
\begin{aligned}
& R(m,n) = R(m)
\otimes_{\cor} R(n), \\
& e(n) = \sum_{\nu \in I^n} e(\nu), \quad e(\beta) = \sum_{\nu \in
I^{\beta}} e(\nu), \\
&  R(\beta) = R(n) e(\beta)=\soplus_{\nu\in I^\beta}R(n)e(\nu), \\
& e(n,i) = \sum_{\nu \in I^{n+1},\; \nu_{n+1} =i}\hs{-3ex} e(\nu)\in
R(n+1), \quad e(i,n)
= \sum_{\nu \in I^{n+1},\; \nu_{1}=i}\hs{-3ex} e(\nu)\in R(n+1), \\
& e(\beta, i) = \sum_{\nu \in I^{\beta + \alpha_i},\; \nu_{n+1} = i}\hs{-3ex}
e(\nu)\in R(\beta+\alpha_i), \quad e(i, \beta) =\hs{-3ex}
 \sum_{\nu \in I^{\beta + \alpha_i},\;\nu_{1}=i}\hs{-3ex} e(\nu)\in R(\beta+\alpha_i).
\end{aligned}
\end{equation}
The algebra $R(\beta)$ is called the {\it {\KLR} at $\beta$}.

Hereafter we will use $\otimes$
instead of $\otimes_{\cor}$ for the sake of simplicity.
By the embedding
$$R(m,n) = R(m) \otimes R(n) \hookrightarrow R(m+n)
\qquad(a\otimes b\longmapsto ab),$$ we regard $R(m,n)$ as a
subalgebra of $R(m+n)$.
For an $R(m)$-module $M$ and an
$R(n)$-module $N$, we define their {\em convolution product} $M\circ
N$ by \eq M\circ N\seteq R(m+n)\otimes_{R(m) \otimes R(n)}(M\otimes
N). \label{def:conv} \eneq Since $R(m+n)$ is a flat module over
$R(m) \otimes R(n)$ (\cite[Proposition 2.16]{KL09}), the bifunctor
$(M,N)\longmapsto M\circ N$ is exact in $M$ and in $N$.

\begin{Prop}[{\cite[Proposition 2.16]{KL09}}] \label{prop:R(n+1)}
We have a decomposition
$$R(n+1) = \bigoplus_{a=1}^{n+1} R(n,1) \tau_{n} \cdots \tau_{a}
= \bigoplus_{a=1}^{n+1} R(n) \otimes \cori[x_{n+1}] \tau_{n} \cdots
\tau_{a}$$ as $R(n,1)$-modules. Here, when $a=n+1$, we understand
$R(n,1)\tau_n \cdots \tau_a = R(n,1)$. In particular, $R(n+1)$ is a
free $R(n,1)$-module of rank $n+1$.
\end{Prop}

\begin{proof}[Sketch of Proof]
Our assertion follows from the right coset decomposition
of $S_{n+1}$:
$$S_{n+1} = \coprod_{a=1}^{n+1} S_{n} s_{n} \cdots s_{a}.$$
\end{proof}

\begin{Prop} \label{prop:R(n)xR(n)}
The $(R(n), R(n))$-bimodule homomorphism
$$R(n) \otimes_{R(n-1)} R(n) \To R(n+1)$$
given by $$x \otimes y \longmapsto x \tau_n y \ \ (x, y \in R(n))$$
is well-defined. Moreover, together with the $(R(n),
R(n))$-bimodule embedding $R(n,1) \hookrightarrow R(n+1)$, it
induces an isomorphism of $(R(n),R(n))$-bimodules
$$R(n) \otimes_{R(n-1)} R(n)\oplus R(n,1)  \overset {\sim} \longrightarrow
R(n+1).$$
\end{Prop}
\begin{proof}
The homomorphism $R(n) \otimes_{R(n-1)} R(n) \longrightarrow R(n+1)$
is well-defined since $\tau_n$ commutes with $R(n-1)$.
It induces a homomorphism
 $$\psi\cl R(n) \otimes_{R(n-1)} R(n) \rightarrow R(n+1)
/ R(n,1)$$ and it is enough to show that $\psi$ is an isomorphism.

Since $R(n) = \bigoplus_{a=1}^{n} \tau_{a} \cdots \tau_{n-1}
\cori[x_{n}] \otimes R(n-1)$, we have
\begin{equation*}
\begin{aligned}
R(n) \otimes_{R(n-1)} R(n) &= \bigoplus_{a=1}^{n} \bigl(\tau_{a} \cdots
\tau_{n-1} \cori[x_{n}] \otimes R(n-1)\bigr) \otimes_{R(n-1)} R(n) \\
& \cong \bigoplus_{a=1}^{n} \tau_{a} \cdots \tau_{n-1}\cori[x_{n}]
\otimes R(n), \\
R(n+1) / R(n,1) &= \dfrac{\bigoplus_{a=1}^{n+1} \tau_{a} \cdots
\tau_{n} \cori[x_{n+1}] \otimes R(n)} {\cori[x_{n+1}] \otimes R(n)}.
\end{aligned}
\end{equation*}
Using
\eqref{eq:partial}, one can verify for $f(x_n)\in\cori[x_n]$, $y\in R(n)$
and $1\le a\le n$
\begin{equation*}
\begin{aligned}
\tau_{a} \cdots \tau_{n-1} f(x_{n}) \tau_{n}y & = \tau_{a}
\cdots \tau_{n-1} (\tau_{n} f(x_{n+1})+\partial_{n} f(x_{n}))y \\
& = \tau_{a} \cdots \tau_{n-1} \tau_{n} f(x_{n+1}) y +
\tau_{a} \cdots \tau_{n-1} \partial_{n}f(x_{n}) y \\
& \equiv \tau_{a} \cdots \tau_{n} f(x_{n+1}) y
\mod \,R(n,1).
\end{aligned}
\end{equation*}
Hence $\psi$ is an isomorphism.
\end{proof}

As an immediate corollary, we obtain:

\begin{Cor} \label{cor:eRe}
There exists a natural isomorphism
\begin{equation*}
\begin{aligned}
& e(n,i) R(n+1) e(n,j) \\
& \simeq
\begin{cases}
q^{-(\alpha_i|\alpha_j)}R(n) e(n-1, j) \otimes_{R(n-1)} e(n-1, i) R(n)
& \text{if $i\neq j$,} \\
q^{-(\alpha_i|\alpha_i)}R(n) e(n-1, i) \otimes_{R(n-1)} e(n-1, i) R(n)
\oplus e(n,i) R(n,1)\, e(n,i)& \text{if $i=j$.}
\end{cases}
\end{aligned}
\end{equation*}
Here, $q$ is the grade-shift functor {\rm(}see \eqref{eq:shift}
below{\rm)}.
\end{Cor}

\begin{proof}
Disregarding the grading, we have by Proposition \ref{prop:R(n)xR(n)}
\begin{equation*}
\begin{aligned}
& e(n,i)R(n+1)e(n,j)  \cong e(n,i) \left(R(n) \otimes_{R(n-1)} R(n)
\oplus R(n,1) \right) e(n,j) \\
& \qquad = R(n) e(n-1, j) \otimes_{R(n-1)} e(n-1, i) R(n) \ \oplus \
e(n,i) R(n,1) e(n,j).
\end{aligned}
\end{equation*}
Our assertion then follows
immediately from $e(n,i)R(n,1)e(n,j) =0$ for $i \neq j$.
\end{proof}

For $\beta \in Q^{+}$, let $\Mod(R(\beta))$ denote the abelian
category of $\Z$-graded $R(\beta)$-modules. Let $q$ denote
the grade-shift functor on $\Mod(R(\beta))$: for a $\Z$-graded
$R(\beta)$-module $M=\bigoplus_{k\in \Z} M_{k}$, we define $qM =
\bigoplus_{k\in \Z} (qM)_{k}$ by \eq &&(qM)_k = M_{k-1} \quad (k \in
\Z). \label{eq:shift} \eneq and we sometimes use the notation
$q_i\seteq q^{(\alpha_i|\alpha_i)/2}$. Thus if $M$ is concentrated
at degree $k$, then $qM$ is concentrated at degree $k+1$.

\begin{Rem}
Let $M$ be a graded $\cor[x]$-module, where $x$ is homogeneous of
degree $a$. Then the multiplication  by $x$ is a morphism $q^{a} M
\overset{x}\longrightarrow M$. It can be understood as a degree
preserving map $q^as\longmapsto xs$ ($s\in M$) by assigning degree
$1$ to $q$.
\end{Rem}

In general, for associative algebras $A$ and $B$, an
$(A, B)$-bimodule $K$ induces a functor $\Phi_{K}\cl \Mod(B)
\rightarrow \Mod(A)$ given by $N \longmapsto K \otimes_{B} N$. In
this case, we say that $K$ is the {\em kernel} of $\Phi_{K}$. Note that
$\Phi_{K}(B) = K \otimes_{B} B \simeq K$,
and hence the kernel is uniquely determined by the functor
$\Phi_K$.

For each $i \in I$, we define the functors
\begin{equation*}
\begin{aligned}
& E_{i}\cl \Mod(R(\beta+\alpha_i)) \longrightarrow \Mod(R(\beta)), \\
& F_{i}\cl \Mod(R(\beta)) \longrightarrow \Mod(R(\beta+ \alpha_i))
\end{aligned}
\end{equation*}
by
\eq \label{eq:E_i,F_i}
\ba{rl}
 E_{i}(N)& = e(\beta, i) N \simeq e(\beta, i)R(\beta+ \alpha_i)
\otimes_{R(\beta+\alpha_i)} N\\[.3ex]
&\hs{11ex} \simeq\Hom_{R(\beta+\alpha_i)}\bl R(\beta+\alpha_i)
e(\beta, i),N\br,
\\[1ex]
 F_{i}(M)& =
M\circ R(\alpha_i)= R(\beta+\alpha_i) e(\beta, i) \otimes_{R(\beta)} M
\ea
\eneq
for $M \in \Mod(R(\beta))$ and $N \in\Mod(R(\beta+\alpha_i))$.

By \cite[Proposition 2.16]{KL09} (see also Proposition \ref{prop:R(n+1)}),
both $E_i$ and $F_i$ are exact functors. Moreover, $F_i$ and $E_i$
are left and right adjoint to each other. That is, for $M \in
\Mod(R(\beta))$ and $N \in \Mod(R(\beta+ \alpha_i))$, there exists a
natural isomorphism
$$\Hom_{R(\beta + \alpha_i)} (F_{i}(M), N)
\isoto \Hom_{R(\beta)} (M, E_{i}(N)).$$ Hence
we obtain adjunction transformations: the counit $\varepsilon\cl
F_{i} \circ E_{i} \longrightarrow \Id$ and the unit
$\eta\cl \Id \longrightarrow E_{i} \circ F_{i}$. We define the
natural transformations
$$x_{E_{i}}\cl E_{i} \rightarrow E_{i} \quad\text{and}\quad x_{F_{i}}\cl  F_{i}
\rightarrow F_{i}$$ as follows:

(a) $x_{E_i}$ is given by the left multiplication by $x_{n+1}$ on
$e(\beta, i) N$ for $N \in \Mod(R(\beta+\alpha_i))$,

(b) $x_{F_i}$ is given by the right multiplication by $x_{n+1}$ on
the kernel $R(\beta+ \alpha_i) e(\beta, i)$ of the functor $F_{i}$,
where $n = |\beta|$.

\vskip 3mm

Then we obtain the following commutative diagram:
\begin{equation} \label{A:comm}
\ba{c}
\xymatrix{ \Hom_{R(\beta+\alpha_i)}(F_{i}(M), N)
\ar[d]^-{x_{F_i}}\ar[r]^-{\sim} & \Hom_{R(\beta)}(M, E_{i}(N))
\ar[d]^-{x_{E_i}}
\\ \Hom_{R(\beta+ \alpha_i)} (F_{i}(M), N)
 \ar[r]^-{\sim} & \Hom_{R(\beta)}(M, E_{i}(N))\,.
 }\ea
\end{equation}

The main properties of the functors $E_{i}$ and $F_{i}$ are given in
the following theorem.

\begin{Thm}\label{thm:E_i}
There exist  natural isomorphisms
\begin{equation*}
E_{i} F_{j} \overset{\sim} \longrightarrow
\begin{cases} q^{-(\alpha_i | \alpha_j)} F_{j} E_{i}\ \ & \text{if}
\ i \neq j, \\
q^{-(\alpha_i| \alpha_i)} F_{i} E_{i} \oplus \Id \otimes \cor[t_i] \ \
& \text{if} \ i=j,
\end{cases}
\end{equation*}
where $t_i$ is an indeterminate of degree $(\alpha_i |
\alpha_i)$ and $\Id \otimes \cor[t_i]\cl
\Mod(R(\beta))\to\Mod(R(\beta))$ is the functor sending $M$
to $M\otimes\cor[t_i]$.
\end{Thm}

\begin{proof}
Our assertion is an immediate consequence of Corollary
\ref{cor:eRe}.
\end{proof}

Let $\xi_n\cl R(n)\to R(n+1)$ be the algebra homomorphism given by
\eq &&\ba{rl}
\xi_n(x_{k})&= x_{k+1} \quad (1 \le k \le n),\\[1.5ex]
\xi_n(\tau_{l})&=\tau_{l+1} \quad (1 \le l \le n-1),\\[1.5ex]
\xi_n(e(\nu))&=\sum_{i\in I}e(i,\nu)\quad (\nu\in I^n). \ea \eneq
Let $R^{1}(n)$ be the image of $\xi_n$. Then $R^1(n)$ is the
subalgebra of $R(n+1)$ generated by $x_{2}, \ldots, x_{n+1}$,
$\tau_{2}, \ldots, \tau_{n}$ and $\xi_n(e(\nu))$ ($\nu\in I^n$),
which is isomorphic to $R(n)$.

\begin{Prop} \label{prop:R^{1}}
The $(R(n),R^1(n))$-bimodule homomorphism
\eq
&&R(n) \otimes_{R^1(n-1)} R^{1}(n) \To R(n+1)
\quad \text{given by $x \otimes y \longmapsto xy$.}
\label{eq:mor1}
\eneq
is well-defined.
This homomorphism is injective and its image $R(n) R^{1}(n)$ has
decompositions
$$R(n) R^{1}(n) = \bigoplus_{a=2}^{n+1} R(n,1) \tau_n \cdots
\tau_{a} = \bigoplus_{a=0}^{n-1} \tau_{a} \cdots
\tau_{1} R(1,n).$$
\end{Prop}

\begin{proof}
It is obvious that \eqref{eq:mor1} is a well-defined
morphism of $(R(n), R^1(n))$-bimodules.
By Proposition \ref{prop:R(n+1)}, we have the decomposition
$$R^1(n)=\bigoplus_{a=2}^{n+1}R^1(n-1)\otimes \cori[x_{n+1}] \tau_{n}\cdots \tau_{a}.$$
Hence we have
\begin{equation*}
\begin{aligned}
R(n) \otimes_{R^1(n-1)} R^{1}(n) & = R(n) \otimes_{R^1(n-1)}
\Bigl(\soplus_{a=2}^{n+1} R^{1}(n-1) \otimes \cori[x_{n+1}] \tau_{n} \cdots
\tau_{a}\Bigr) \\
& \cong \bigoplus_{a=2}^{n+1} R(n) \otimes \cori[x_{n+1}] \tau_{n}
\cdots \tau_{a} = \bigoplus_{a=2}^{n+1} R(n,1) \tau_{n} \cdots
\tau_{a}.
\end{aligned}
\end{equation*}
Similarly, we have $$R(n+1) = \bigoplus_{a=0}^{n} \tau_{a}
\cdots \tau_{1} R(1, n)$$ and $$R(n) \otimes_{R^1(n-1)} R^{1}(n) \cong
\bigoplus_{a=0}^{n-1} \tau_{a}\cdots \tau_{1} R(1,n).$$ Now our
assertions follow immediately.
\end{proof}

By Proposition \ref{prop:R^{1}}, there exists a map $\varphi_1\cl
R(n+1) \rightarrow R(n) \otimes\cori[x_{n+1}] $ given by

\begin{equation}\label{eq:phi-1}
\begin{aligned}
R(n+1) & \rightarrow \Coker\bigl(R(n) \otimes_{R^1(n-1)} R^{1}(n)\to R(n+1)\bigr)\\
&\cong \dfrac{\bigoplus_{a=1}^{n+1} R(n,1) \tau_{n} \cdots \tau_{a}}
{\bigoplus_{a=2}^{n+1} R(n,1)
\tau_{n} \cdots \tau_{1}}
  \isofrom R(n,1) \tau_{n} \cdots \tau_{1} \\
&\isofrom R(n) \otimes \cori[x_{n+1}].
\end{aligned}
\end{equation}
 Similarly, there is another map $\varphi_2\cl R(n+1)
\rightarrow \cori[x_{1}] \otimes R^{1}(n) $ given by
\begin{equation}\label{eq:phi-2}
\begin{aligned}
R(n+1) \rightarrow &\Coker\bigr(R(n) \otimes_{R^1(n-1)} R^{1}(n)\to
R(n+1)\bigr)
\\&
\cong \dfrac{\bigoplus_{a=0}^{n} \tau_{a} \cdots \tau_{1} R(1,n)}
{\bigoplus_{a=0}^{n-1} \tau_{a} \cdots
\tau_{1} R(1,n)}
 \isofrom \tau_{n} \cdots \tau_{1} R(1,n) \\
 & \isofrom \cori[x_{1}] \otimes R^{1}(n).
\end{aligned}
\end{equation}
Note that
\begin{equation*}
\begin{aligned}
& x_k \tau_n \cdots \tau_1 \equiv \tau_n \cdots \tau_1 x_{k+1} \ \
(1 \le k \le n), \\
& \tau_{l} \tau_{n} \cdots \tau_{1} \equiv \tau_{n} \cdots \tau_{1}
\tau_{l+1} \ \ (1 \le l \le n-1), \\
& x_{n+1} \tau_{n} \cdots \tau_{1} \equiv \tau_{n} \cdots \tau_{1}
x_{1}\hs{6ex}\mod R(n)R^{1}(n).
\end{aligned}
\end{equation*}

We shall identify two algebras $R(n) \otimes \cori[x_{n+1}]$ and
$\cori[x_{1}] \otimes R^{1}(n)$, and write $$R(n) \otimes
\cori[x_{n+1}]= \cori[x_{1}] \otimes R^{1}(n) = R(n) \otimes
\cori[t]$$ for some indeterminate $t$.
Then for any $a\otimes f(t)\in  R(n) \otimes\cori[t]$, we have
$$a f(x_{n+1})\tau_n \cdots \tau_1
\equiv \tau_n \cdots \tau_1\; \xi_n(a)\otimes f(x_1)
\mod R(n)R^{1}(n).$$
Hence $\varphi_1$ and $\varphi_2$ coincide and we obtain:

\begin{Cor} \label{cor:varphi-1}
There is an exact sequence of $(R(n), R(n))$-bimodules
$$0 \longrightarrow R(n) \otimes_{R^1(n-1)} R^{1}(n) \longrightarrow
R(n+1) \overset {\varphi} \longrightarrow R(n) \otimes \cor[t]
\longrightarrow 0, $$ where the map $\varphi$ is given by
\eqref{eq:phi-1} or \eqref{eq:phi-2}. Here, the right $R(n)$-module
structure on $R(n+1)$ is given by the embedding $\xi_{n}\cl
R(n)\overset{\sim} \longrightarrow R^1(n)\hookrightarrow R(n+1)$.
Moreover, both the left multiplication by $x_{n+1}$ and the
right multiplication by $x_1$ on $R(n+1)$ are compatible with the
multiplication by $t$ on $R(n) \otimes \cor[t]$.
\end{Cor}
\begin{proof}
By the construction of $\varphi$, our assertions follow
immediately.
\end{proof}

\vskip 3mm

For each $i \in I$, we define the functor
$$\bF_{i} \cl\Mod(R(\beta)) \longrightarrow
\Mod(R(\beta+\alpha_i))$$ by
\begin{equation} \label{eq:E_ibar}
\bF_{i}(M) =R(\alpha_i)\circ M=R(\beta+\alpha_i) e(i, \beta)
\otimes_{R(\beta)} M,
\end{equation}
where the right $R(\beta)$-module structure on $R(\beta+\alpha_i)
e(i, \beta)$ is given by the embedding
$$R(\beta) \overset{\sim} \longrightarrow  R^{1}(\beta) \hookrightarrow R(\beta
+ \alpha_i).$$

\begin{Thm}\label{thm:bF}\hfill
\bna
\item There exists a natural isomorphism
$$ \bF_{j} E_{i} \overset{\sim} \longrightarrow E_{i}
\bF_{j} \ \ \text{for} \ i \neq j.$$
\item We have an exact sequence
in $\Mod(R(\beta))${\rm:}
$$0 \rightarrow \bF_{i} E_{i}M \longrightarrow E_{i}\bF_{i}M
\longrightarrow q^{-(\alpha_i | \beta)} M\otimes \cor[t_i]
\rightarrow 0$$ which is functorial in $M\in\Mod(R(\beta))$.  Here
$t_i$ is an indeterminate of degree $(\alpha_i | \alpha_i)$.
\end{enumerate}
\end{Thm}
\begin{proof}
The first assertion follows from
$$e(\beta,i)\bl R(n,1) \tau_{n} \cdots \tau_{1} \br
e(j,\beta)=0 \quad\text{for $i\not=j$ (see \eqref{eq:phi-1}).}$$
(b) is an immediate consequence of Corollary \ref{cor:varphi-1}.
\end{proof}

\vskip 3em


\section{The cyclotomic Khovanov-Lauda-Rouquier algebras} \label{sec:RLambda}

Let $\Lambda \in P^{+}$ be a dominant integral weight. In this
section, we study the structure of the cyclotomic
Khovanov-Lauda-Rouquier algebra $R^{\Lambda}$ and the functors
$E_{i}^{\Lambda}$, $F_{i}^{\Lambda}$ defined on the category of
$R^{\Lambda}$-modules.

\subsection{Definition of cyclotomic \KLRs}

For $\Lambda\in P^+$ and $i\in I$, choose a monic polynomial of
degree $\langle h_{i},\Lambda \rangle$ \eq
a_i^\Lambda(u)=\sum_{k=0}^{\langle h_{i},\Lambda \rangle}
c_{i;k}u^{\langle h_{i},\Lambda \rangle-k} \eneq with $c_{i;k}\in
\cor_{k(\alpha_i|\alpha_i)}$ and $c_{i;0}=1$.

For $1\le k\le n$, we define
\begin{equation}
\x[k]= \sum_{\nu \in I^{n}} a_{\nu_k}^\Lambda(x_k) e(\nu)\in R(n).
\end{equation}
Hence $\x[k]e(\nu)$ is a homogeneous element of $R(n)$ with degree
$2(\alpha_{\nu_k}|\Lambda)$.

\begin{Def} \label{def:RLambda}
The {\em cyclotomic Khovanov-Lauda-Rouquier
algebra $R^{\Lambda}(\beta)$ at $\beta$} is defined to be the
quotient algebra
$$R^{\Lambda}(\beta) = \dfrac{R(\beta)}  {R(\beta)\x R(\beta)}.$$
Here, we understand $R^\Lambda(\beta)=\cor$ for $\beta=0$.
\end{Def}

For each $n\ge0$, we also define
\begin{equation} \label{eq:x(n)}
R^{\Lambda}(n) = \dfrac{R(n)} {R(n)
\x R(n)}\cong\soplus_{|\beta|=n}R^\Lambda(\beta).
\end{equation}
Then we may write $$R^{\Lambda}(\beta) = R^{\Lambda}(n)e(\beta),
\quad \text{where}\quad e(\beta) = \sum_{\nu \in I^{\beta}}
e(\nu).$$

Now we will prove that $R^\Lambda(n)$ is a finitely generated
$\cor$-module (cf.\ \cite{KL09, BK09}).
\begin{Lem}\label{lem:van}
Let $M$ be an $R(n)$-module, $f\in \cor[x_1,\ldots,x_n]$ and
$\nu\in I^n$ such that $\nu_{n-1}=\nu_n$. Then $fe(\nu)M=0$ implies
$(\partial_{n-1}f)e(\nu)M=0$ and $(s_{n-1}f)e(\nu)M=0$.
\end{Lem}
\Proof
The equality
\eqn
(x_{n-1}-x_{n})\tau_{n-1}f\tau_{n-1}e(\nu)
&=&(x_{n-1}-x_n)\bigl((s_{n-1}f)\tau_{n-1}
+\partial_{n-1} f(x)\bigr)\tau_{n-1}e(\nu)\\
&=&(s_{n-1}f-f)\tau_{n-1}e(\nu)
=(\tau_{n-1}f-\partial_{n-1}f-f\tau_{n-1})e(\nu) \eneqn implies
$(\partial_{n-1}f)e(\nu)M=0$. The last equality follows from
$$(x_{n-1}-x_n)\partial_{n-1}f=s_{n-1}f-f.$$ \QED

\begin{Lem}\label{lem:integrable}
Let $\beta\in Q^+$ with $|\beta|=n$ and $i\in I$.
\bna
\item There exists a monic polynomial $g(u)$ with coefficients in $\cor$
such that $g(x_a)=0$ in $R^\Lambda(\beta)$
for any $a$ $(1\le a\le n)$.
\item There exists $m$ such that
$R^\Lambda(\beta+k\alpha_i)=0$ for any $k\ge m$.
\ee
\end{Lem}
\Proof
(a) By induction on $a$, it is enough to show that
\eq &&\parbox{\my}{ for an $R(n)$-module $M$ and a
monic polynomial $g(u) \in \cor[u]$ such that $g(x_a)M=0$,
there exists a monic polynomial $h(u) \in \cor[u]$ such that
$h(x_{a+1})M=0$.}
\label{eq:x12} \eneq

\smallskip
We shall prove that, for any $\nu\in I^n$, we can find a monic
polynomial $h(u) \in \cor[u]$ such that $h(x_{a+1})e(\nu)M=0$.

\noindent
(i)\ If $\nu_a\not=\nu_{a+1}$, then we have
$$g(x_{a+1})Q_{\nu_a,\nu_{a+1}}(x_a,x_{a+1})e(\nu)M=g(x_{a+1})\tau_a^2e(\nu)M=
\tau_ag(x_a)\tau_ae(\nu)M=0.$$
Since
$Q_{\nu_a,\nu_{a+1}}(x_a,x_{a+1})$ is a monic polynomial in $x_{a+1}$
with coefficients in $\cor[x_a]$
(up to an invertible constant multiple),
there exists a monic polynomial $h(x_{a+1})$ such that
$$h(x_{a+1})\in
\cor[x_a,x_{a+1}]g(x_a)
+\cor[x_a,x_{a+1}]g(x_{a+1})Q_{\nu_a,\nu_{a+1}}(x_a,x_{a+1}).$$
Then $h(x_{a+1})e(\nu)M=0$.

\smallskip
\noindent
(ii)\ The case $\nu_a=\nu_{a+1}$ immediately follows from Lemma~\ref{lem:van}.

\medskip
\noindent
(b)\ For $\nu\in I^n$, we set
$\Supp_i(\nu)\seteq\#\set{k}{\text{$1\le k\le n$ and $\nu_k=i$}}$,
and $\nu_{\le m}\seteq(\nu_1,\ldots,\nu_m)$,
$\nu_{\ge m}\seteq(\nu_{m},\ldots,\nu_n)$ for $1\le m\le n$.

 In order to see (b), it is enough to show
\eq
&\parbox{70ex}{for a given $\Lambda$ and $n$, there exists $m>0$
such that $e(\nu)R^\Lambda(n+m)=0$
for any $\nu\in I^{n+m}$ with $\Supp_i(\nu)\ge m$.}
\label{eq:25}
\eneq

In order to prove this we first show the following:
\eq
&&\parbox{70ex}{for a given $\Lambda$ and $n$, there exists $m>0$
such that $e(\nu)R^\Lambda(n+m)=0$
for any $\nu\in I^{n+m}$ with $\nu_a=i$ ($n<a\le n+m$).}
\label{eq:30}
\eneq
By (a), there exists a monic polynomial $g(u)$ of degree $m\ge0$ such that
$g(x_n)R^\Lambda(n)=0$.
Hence $g(x_n)R^\Lambda(n+k)=0$ for any $k\ge0$.
The repeated application of Lemma~\ref{lem:van}
implies
$$\bl\partial_{n+m-1}\cdots\partial_{n}g(x_n)\br e(\nu)R^\Lambda(n+k)=0
\quad\text{for $k\ge m$.}$$
 Since $\partial_{n+m-1}\cdots\partial_{n}g(x_n)=\pm1$, we
obtain $e(\nu)R^\Lambda(n+k)=0$ for $k\ge m$. Thus we obtain \eqref{eq:30}.

\smallskip
Now we shall show \eqref{eq:25} by induction on $n$.
If $n=0$, then  \eqref{eq:25} immediately follows from \eqref{eq:30}.
Assume $n>0$. By the induction hypothesis, there exists $k>0$ such that
\eq
&&\text{$e(\nu)R^\Lambda(n-1+k)=0$
for any $\nu\in I^{n-1+k}$ with $\Supp_i(\nu)\ge k$.}
\label{eq:31}
\eneq
Applying \eqref{eq:30} to $n-1+k$, there exists $m$
such that
\eq
&&\parbox{75ex}{$e(\nu)R^\Lambda(n-1+k+m)=0$
if  $\nu\in I^{n-1+k+m}$  satisfies $\nu_a=i$
for any $a$ such that $n-1+k<a\le n-1+k+m$.}
\label{eq:32}
\eneq
Now let $\nu\in  I^{n-1+k+m}$ such that $\Supp_i(\nu)\ge -1+k+m$.
We shall show $e(\nu)R^\Lambda(n-1+k+m)=0$.
If $\Supp_i(\nu_{\le n-1+k})\ge k$, then \eqref{eq:31} implies
$e(\nu)R^\Lambda(n-1+k+m)=0$. If $\Supp_i(\nu_{\le n-1+k})\le k-1$,
then
$\Supp_i(\nu_{\ge n+k})=\Supp_i(\nu)-\Supp_i(\nu_{\le n-1+k})
\ge(-1+k+m)-(k-1)=m$, which implies that
$\nu_a=i$ for $n+k\le a\le n-1+k+m$.
Hence \eqref{eq:32} implies $e(\nu)R^\Lambda(n-1+k+m)=0$.
\QED

\begin{Cor}\label{lem:cyclofinite}
For any $\beta\in Q^+$, $R^\Lambda(\beta)$ is a finitely generated
$\cor$-module.
\end{Cor}

\subsection{Exactness of $F_i^\Lambda$}
For each $i\in I$, we define the functors
\begin{equation*}
\begin{aligned}
& E_{i}^{\Lambda}\cl \Mod(R^{\Lambda}(\beta+\alpha_i)) \longrightarrow \Mod(R^{\Lambda}(\beta)), \\
& F_{i}^{\Lambda}\cl \Mod(R^{\Lambda}(\beta)) \longrightarrow
\Mod(R^{\Lambda}(\beta+ \alpha_i))
\end{aligned}
\end{equation*}
by
\begin{equation} \label{eq:E_iLambda}
\begin{aligned}
& E_{i}^{\Lambda}(N) = e(\beta, i) N = e(\beta, i)R^{\Lambda}(\beta+
\alpha_i) \otimes_{R^{\Lambda}(\beta+\alpha_i)} N, \\
& F_{i}^{\Lambda}(M) = R^{\Lambda}(\beta+\alpha_i) e(\beta, i)
\otimes_{R^{\Lambda}(\beta)} M,
\end{aligned}
\end{equation}
where $M \in \Mod(R^{\Lambda}(\beta))$, $N \in
\Mod(R^{\Lambda}(\beta+\alpha_i))$.

The purpose of this section is to prove
the following theorem.

\begin{Thm}\label{th:proj}
The module $R^{\Lambda}(\beta+\alpha_i)e(\beta, i)$
is a projective right $R^\Lambda(\beta)$-module.
Similarly, $e(\beta, i)R^{\Lambda}(\beta+\alpha_i)$
is a projective left $R^\Lambda(\beta)$-module.
 \end{Thm}
Of course, the second statement is a consequence of the first
by the anti-involutions of \KLRs.
As its immediate consequence, we have
\begin{Cor}\label{cor:exact} \hfill
\bnum
\item
The functor $E_{i}^{\Lambda}$ sends finitely generated
projective modules to finitely generated projective modules.

\item The functor $F_{i}^{\Lambda}$ is exact.
\enum
\end{Cor}

We shall prove Theorem~\ref{th:proj} as a consequence of the following theorem.

\begin{Thm}\label{th:main}
For any $i\in I$ and $\beta\in Q^+$, there exists an exact sequence
of $R(\beta+\alpha_i)$-modules \eq &&0\To
q^{(\alpha_i|2\Lambda-\beta)}\bF_iM\To F_iM\To F^\Lambda_iM\To0
\label{eq:fbarf}
\eneq which is functorial in $M\in\Mod(R^\Lambda(\beta))$.
\end{Thm}

\subsection{Proof of Theorem \ref{th:proj} and Theorem \ref{th:main} }

Set
\begin{equation}\label{eq:kernel}
\begin{aligned}
& F^{\Lambda} = R^{\Lambda}(\beta+\alpha_i) e(\beta, i) =
\dfrac{R(\beta+\alpha_i) e(\beta, i)}{R(\beta+\alpha_i)
\x R(\beta+\alpha_i) e(\beta, i)}, \\
& K_{0}= R(\beta+ \alpha_i) e(\beta, i) \otimes_{R(\beta)}
R^{\Lambda}(\beta) = \dfrac{R(\beta+\alpha_i) e(\beta,
i)}{R(\beta+\alpha_i)
\x R(\beta) e(\beta, i)}, \\
& K_{1} = R(\beta+\alpha_i) e(i, \beta) \otimes_{R(\beta)}
R^{\Lambda}(\beta) = \dfrac{R(\beta+\alpha_i) e(i,
\beta)}{R(\beta+\alpha_i) \x[2] R^{1}(\beta) e(i, \beta)}.
\end{aligned}
\end{equation}
We regard $R(\beta+\alpha_i) e(i, \beta)$ as a right
$R(\beta)$-module through the embedding
$$R(\beta)\isoto
R^1(\beta)\hookrightarrow R(\beta+\alpha_i),$$
which defines a right $R^{\Lambda}(\beta)$-module structure on $K_1$. Hence
$F^\Lambda$, $K_0$ and $K_1$ can be regarded as
$\bl R(\beta+\alpha_i),R^\Lambda(\beta)\br$-bimodules. Then
$F^{\Lambda}$, $K_{0}$ and $K_1$ are the kernels of the functors
$F_i^\Lambda$, $F_i$ and $\bF_i$ from $\Mod(R^\Lambda(\beta))$ to
$\Mod(R(\beta+\alpha_i))$, respectively.

Let $t_i$ be an indeterminate of degree $(\alpha_i|\alpha_i)$. Then,
$\cor[t_i]$ acts from the right on $R(\beta+\alpha_i) e(i, \beta)$
and $K_1$  by $t_i = x_1 e(i, \beta)$.
Hence $R(\beta+\alpha_i) e(i, \beta)$ and $K_{1}$ have a structure
of $(R(\beta+\alpha_i),R(\beta)\otimes\cor[t_i])$-bimodule.
Similarly, $\cor[t_i]$ acts from the right on $K_{0}$ and
$F^\Lambda$ by $t_i=x_{n+1} e(\beta, i)$, and $K_{0}$ and
$F^\Lambda$ also have a structure of
 $(R(\beta+\alpha_i),R(\beta)\otimes\cor[t_i])$-bimodule.
Note that $F^\Lambda$, $K_1$ and $K_0$ are in fact
$(R(\beta+\alpha_i),R^\Lambda(\beta)\otimes\cor[t_i])$-bimodules.

\begin{Lem}[\cite{KL09}]\label{lem:Kproj}
Both $K_1$ and $K_0$ are finitely generated projective right
$\bl R^\Lambda(\beta)\otimes\cor[t_i]\br$-modules.
\end{Lem}
\Proof
The statement for $K_0$ follows from
$ K_{0}= R(\beta+ \alpha_i) e(\beta, i) \otimes_{R(\beta)\otimes\cor[t_i]}
\bigl(R^{\Lambda}(\beta)\otimes\cor[t_i]\bigr)$
and the fact that
$R(\beta+ \alpha_i) e(\beta, i)$ is a finitely generated projective module
over $R(\beta)\otimes\cor[t_i]$ (Proposition~\ref{prop:R(n+1)}).
The proof for $K_1$ is similar.
\QED

\begin{Lem} \label{lem:R(beta+alpha_i)}
For $i \in I$ and $\beta \in Q^{+}$ with $|\beta|=n$, we have
\bna
\item $R(\beta+\alpha_i) \x R(\beta+\alpha_i) =
\sum_{a=0}^{n} R(\beta+\alpha_i) \x\tau_1 \cdots
\tau_{a}.$

\item $R(\beta+\alpha_i)\x R(\beta+\alpha_i)
e(\beta, i) = R(\beta+\alpha_i) \x R(\beta) e(\beta, i)$\\
\hs{40ex}$+R(\beta+\alpha_i) \x \tau_{1}\cdots \tau_{n} e(\beta, i)$.
\end{enumerate}
\end{Lem}
\begin{proof}
The assertion (a) can be verified easily from
$$R(n+1) = \sum_{a=0}^{n} \cori[x_{1}] \otimes
R^{1}(n) \tau_{1} \cdots \tau_{a}.$$ The assertion (b) follows
immediately from (a).
\end{proof}

Let $\pi\cl  K_{0}\longrightarrow F^{\Lambda}$ be the canonical projection.
We will construct a short exact sequence of
$(R(\beta+\alpha_i),R^\Lambda(\beta)\otimes\cor[t_i])$-bimodules
$$0 \longrightarrow K_{1}
\overset{P} \longrightarrow K_{0} \overset {\pi} \longrightarrow
F^{\Lambda} \longrightarrow 0.$$

\medskip
Let $\widetilde{P}\cl R(\beta+\alpha_i) e(i, \beta) \longrightarrow
K_{0}$ be the right multiplication by $\x \tau_{1} \cdots \tau_{n}$.
Then $\widetilde{P}$ is a left $R(\beta+\alpha_i)$-linear
homomorphism.

Using \eqref{eq:kernel} and Lemma \ref{lem:R(beta+alpha_i)}, we see
that
\eq&&
\on{Im} (\widetilde{P})=
\Ker \pi = \dfrac{R(\beta+\alpha_i) \x
R(\beta + \alpha_i) e(\beta, i)} {R(\beta + \alpha_i)
\x R(\beta) e(\beta, i)}\subset K_0.
\label{eq:Ppi}
\eneq

\begin{Lem}
The map $\widetilde{P}\cl R(\beta+\alpha_i) e(i,\beta)
\longrightarrow K_{0}$ is a right $(R(\beta) \otimes
\cor[t_i])$-module homomorphism.
\end{Lem}
\begin{proof}
For $1 \le a \le n$, we have
\begin{equation*}
\begin{aligned}
&x_{a+1} \bl\x\tau_{1} \cdots \tau_{n}e(\beta,i)\br  =
\x \tau_{1} \cdots \tau_{ a-1} (x_{ a+1} \tau_{a})
\tau_{a+1} \cdots \tau_{n} e(\beta,i)\\
&\hs{10ex} =\x \tau_{1} \cdots
\tau_{a-1}(\tau_{ a} x_{ a}+ e_{a,a+1}) \tau_{a+1}
\cdots \tau_{n} e(\beta,i) \\
&\hs{10ex} = \bl\x\tau_{1} \cdots \tau_{n}e(\beta,i)\br x_{a}+
(\x \tau_{1} \cdots \tau_{a-1} \tau_{a+1} \cdots
\tau_{n}) e_{a, n+1}e(\beta,i)\\
&\hs{10ex} = \bl\x \tau_{1} \cdots \tau_{n}e(\beta,i)\br x_{a}  +(\tau_{a+1}
\cdots \tau_{n}) (\x \tau_{1} \cdots
\tau_{a-1})e_{a,n+1}e(\beta,i)\\
&\hs{10ex} \equiv \bl\x \tau_{1} \cdots \tau_{n} e(\beta,i)\br x_{a}
\mod R(\beta+ \alpha_i) \x R(\beta) e(\beta,i),
\end{aligned}
\end{equation*}
and
\begin{equation*}
\begin{aligned}
&x_{1} \bl\x\tau_{1} \cdots \tau_{n}  e(\beta,i)\br  = \x
(\tau_{1} x_{2} - e_{1,2}) \tau_{2} \cdots \tau_{n}  e(\beta,i)\\
&\hs{10ex} = \x \tau_{1} (x_{2} \tau_{2}) \tau_{3} \cdots
\tau_{n}e(\beta,i)- (\tau_{2} \cdots \tau_{n}) \x e_{1,n+1} e(\beta,i)\\
&\hs{10ex} \equiv \x \tau_{1} (x_{2} \tau_{2}) \tau_{3} \cdots\tau_{n}e(\beta,i)
=\x \tau_{1} (\tau_{2} x_{3} - e_{2,3})
\tau_{3} \cdots \tau_{n}  e(\beta,i)\\
&\hs{10ex} = \x \tau_{1} \tau_{2} (x_{3} \tau_{3}) \tau_{4}
\cdots \tau_{n} e(\beta,i)- (\tau_{3} \cdots \tau_{n}) \x \tau_1
e_{2,n+1} e(\beta,i)
\\
&\hs{10ex} \ \ \cdots \cdots \\
&\hs{10ex} \equiv (\x \tau_{1} \cdots \tau_{n}) x_{n+1} e(\beta,i)
\mod R(\beta+ \alpha_i) \x R(\beta) e(\beta,i).
\end{aligned}
\end{equation*}
For $1 \le a \le n-1$, we have
\begin{equation*}
\begin{aligned}
& \tau_{a+1} \bl\x \tau_{1} \cdots \tau_{n}e(\beta,i)\br =
\x \tau_{1} \cdots \tau_{a-1} (\tau_{a+1} \tau_{a}
\tau_{a+1}) \tau_{a+2} \cdots \tau_{n}  e(\beta,i)\\
& = \x \tau_{1} \cdots \tau_{a-1} \bl\tau_{a} \tau_{a+1}
\tau_{a}+ \overline{Q}_{a,a+1,a+2}\br \tau_{a+2} \cdots \tau_{n}  e(\beta,i)\\
& = \bl\x \tau_{1} \cdots \tau_{a-1} \tau_{a} \tau_{a+1}
\cdots \tau_{n}e(\beta,i)\br \tau_{a}+ \x \tau_{1} \cdots
\tau_{a-1} \overline{Q}_{a,a+1,a+2} \tau_{a+2} \cdots \tau_{n}e(\beta,i).
\end{aligned}
\end{equation*}
Since
\begin{equation*}
\begin{aligned}
\x & \tau_{1} \cdots \tau_{a-1} \cor[x_{1}, \ldots,
x_{n+1}] \tau_{a+2} \cdots \tau_{n} e(\beta,i)\\
& \subset \x \tau_{1} \cdots \tau_{a-1} \sum_{ w \in
\langle s_{a+2}, \ldots, s_{n}
\rangle } \tau_{w} \cor[x_{1}, \ldots, x_{n+1} ] e(\beta,i)\\
& \subset \sum_{ w \in \langle s_{a+2}, \ldots, s_{n} \rangle }
\tau_{w} \x \tau_{1} \cdots \tau_{a-1} \cor[x_{1},
\ldots, x_{n+1}]  e(\beta,i)\\
& \subset \sum_{ w \in \langle s_{a+2}, \ldots, s_{n} \rangle }
\tau_{w}\cor[x_{n+1}] \x \tau_{1} \cdots \tau_{a-1} \cor[x_{1},
\ldots, x_{n}] e(\beta,i)\\
& \subset R(\beta+\alpha_i)\x R(\beta) e(\beta,i),
\end{aligned}
\end{equation*}
we are done.
\end{proof}

Since $\widetilde{P}$ is
$\bl R(\beta+\alpha_i),R(\beta)\otimes\cor[t_i]\br$-bilinear,
it maps $R(\beta +\alpha_{i})\x[2] R^{1}(\beta) e(i, \beta)$ to
$\dfrac{R(\beta +\alpha_{i}) \x R(\beta) e(\beta,i)}
{R(\beta +\alpha_{i})\x R(\beta) e(\beta,i)}=0$ in
$K_{0}$. Hence we get a $(R(\beta+\alpha_i),R(\beta)\otimes\cor[t_i])$-bilinear
homomorphism
\eqn&&P\cl K_{1} \rightarrow K_{0}\eneqn
given by the right multiplication by $\x \tau_{1}
\cdots \tau_{n}$. Moreover, \eqref{eq:Ppi} implies
$\on{Im}\, {P} =\Ker \pi$.
Therefore we get an exact sequence of
$(R(\beta+\alpha_i), R(\beta)\otimes\cor[t_i])$-bimodules
\eq&&K_{1} \overset{P} \longrightarrow K_{0} \overset{\pi} \longrightarrow
F^{\Lambda}\longrightarrow  0.\eneq

We will show next that $P$ is injective by
constructing a homomorphism $Q\cl K_0\to K_1$ nearly inverse to $P$.

For $1\le a\le n$, we define an element $g_a$ of $R(\beta+\alpha_i)$
by
\begin{equation} \label{eq:ga}
g_{a} =\tau_{a}\kern-1ex \sum_{\substack{\nu\in I^{\beta+\alpha_i},\\\nu_{a} \neq
\nu_{a+1}}}e(\nu) + \bigl(x_{a+1} - x_{a} -(x_{a+1} - x_{a})^2
\tau_{a}\bigr)\kern-1ex\sum_{\substack{\nu \in I^{\beta+\alpha_i},\\ \nu_{a}=\nu_{a+1}}}
e(\nu)\in R(\beta+\alpha_i).
\end{equation}

\begin{Rem}\label{rem:intertwiner}
The elements $$\bigl(1+(x_a-x_{a+1})\tau_a\bigr)e_{a,a+1}=
\bigl(\tau_a(x_{a+1}-x_{a})-1\bigr)e_{a,a+1}$$ are called the {\it
intertwiners}. Note that they satisfy the same  identities
as \eqref{eq:intertwiner} given below. The elements $g_a$'s are
variants of intertwiners.
\end{Rem}

\begin{Lem} \label{lem:ga}
For $1 \le a \le n$, we have
\begin{equation}
x_{s_a(b)}g_{a} = g_{a}x_{b} \  (1\le b\le n+1) \quad\text{and}\quad
\tau_{a} g_{a+1} g_{a} = g_{a+1} g_{a} \tau_{a+1}.
\label{eq:intertwiner}
\end{equation}
\end{Lem}

\begin{proof}
We will first show $x_{a} g_{a} e(\nu) = g_{a}
x_{a+1} e(\nu)$. If $\nu_a \neq \nu_{a+1}$, our assertion is clear.
If $\nu_a=\nu_{a+1}$,
then we have
$$(x_{a} g_{a} - g_{a} x_{a+1})e(\nu) = \bigl((x_{a+1} - x_{a})(x_{a} - x_{a+1}) +
(x_{a+1} - x_{a})^2\bigr)e(\nu) =0.$$ Similarly, we have $x_{a+1} g_{a} = g_{a}
x_{a}$, and $x_bg_a=g_ax_b$ for $b\not=a,a+1$.

In order to prove the last identity, set $S=\tau_{a} g_{a+1} g_{a} -
g_{a+1} g_{a} \tau_{a+1}$.  We have by \eqref{eq:partial}
{\allowdisplaybreaks
\begin{equation*}
\begin{aligned}
 (\tau_{a} g_{a+1} g_{a}) x_{a} & = \tau_{a} g_{a+1} x_{a+1} g_{a}
=\tau_{a} x_{a+2} g_{a+1} g_{a} = x_{a+2} \tau_{a} g_{a+1} g_{a},\\
 (g_{a+1} g_{a} \tau_{a+1}) x_{a}& = g_{a+1} g_{a} x_{a} \tau_{a+1}
= g_{a+1} x_{a+1} g_{a} \tau_{a+1} = x_{a+2} g_{a+1}
g_{a}\tau_{a+1}, \\
 (\tau_{a} g_{a+1} g_{a}) x_{a+1} & = \tau_{a} g_{a+1} x_{a} g_{a} =
\tau_{a} x_{a} g_{a+1} g_{a} = (x_{a+1} \tau_{a} -e_{a,a+1})
g_{a+1}g_{a} \\
& = x_{a+1} \tau_{a} g_{a+1} g_{a} - g_{a+1} g_{a} e_{a+1, a+2}, \\
(g_{a+1} g_{a} \tau_{a+1}) x_{a+1} & = g_{a+1} g_{a} ( x_{a+2}
\tau_{a+1} -e_{a+1, a+2}) \\
& =x_{a+1} g_{a+1} g_{a} \tau_{a+1} - g_{a+1} g_{a} e_{a+1, a+2},\\
 (\tau_{a} g_{a+1} g_{a}) x_{a+2} & = \tau_{a} x_{a+1}g_{a+1} g_{a}
= (x_{a} \tau_{a}+e_{a,a+1})g_{a+1}g_{a} \\
& = x_{a} \tau_{a} g_{a+1} g_{a} + g_{a+1} g_{a} e_{a+1, a+2}, \\
(g_{a+1} g_{a} \tau_{a+1}) x_{a+2} & = g_{a+1} g_{a} ( x_{a+1}
\tau_{a+1}+e_{a+1, a+2}) \\
& =x_{a} g_{a+1} g_{a} \tau_{a+1} + g_{a+1} g_{a} e_{a+1, a+2}.
\end{aligned}
\end{equation*}}
Hence $Sx_{b} = x_{s_{a, a+2}(b)}S$ for all $b$.

On the other hand, we can easily see that $S$
does not contain the term $\tau_{a} \tau_{a+1}\tau_{a}$ and
belongs to the $\cor[x_a,x_{a+1},x_{a+2}]$-module generated by
$1$, $\tau_a$, $\tau_{a+1}$, $\tau_a\tau_{a+1}$ and
$\tau_{a+1}\tau_a$. It follows that
$S=0$.
\end{proof}

\begin{Prop}
Let $\widetilde{Q}\cl R(\beta+ \alpha_i) e(\beta, i)\longrightarrow
K_{1}$ be the $R(\beta+\alpha_i)$-module homomorphism given by the
right multiplication by $g_{n}\cdots g_{1}$. Then $\widetilde{Q}$ is
right $R(\beta)\otimes \cor[t_i]$-linear. That is, \eqn
&&\widetilde{Q}(sx_a) = \widetilde{Q}(s) x_{a+1} \ (1 \le a \le
n),\\
&& \widetilde{Q}(sx_{n+1}) = \widetilde{Q}(s) x_{1}, \\
&&\widetilde{Q}(s\tau_a)=\widetilde{Q}(s)\tau_{a+1}\ (1\le a<n)\eneqn
for any $s\in  R(\beta+ \alpha_i) e(\beta, i)$.
\end{Prop}
\begin{proof}
It follows immediately from the preceding lemma.
\end{proof}

\begin{Prop}
The map $\widetilde{Q}$ induces a well-defined $(R(\beta+\alpha_i),
R(\beta)\otimes\cor[t_i])$-bimodule homomorphism
$$Q\cl K_{0} = \dfrac{R(\beta+\alpha_i) e(\beta, i)} {R(\beta+\alpha_i)
\x R(\beta) e(\beta, i)}
\longrightarrow K_{1} =\dfrac {R(\beta+\alpha_i) e(i, \beta)}
{R(\beta+\alpha_i) \x[2] R^{1}(\beta) e(i, \beta)}$$ given
by the right multiplication by $g_{n} \cdots g_{1}$.
\end{Prop}

\begin{proof}
Since $$R(\beta+ \alpha_i) \x R(\beta) = \bigoplus_{a=0}^{n-1}
R(\beta+\alpha_i) \x \tau_{1} \cdots \tau_{a},$$ it suffices to show
that $\widetilde{Q}$ sends $\x \tau_{1} \cdots \tau_{a}e(\beta,i)$
($0\le a\le n-1$) to $0$ in $K_1$. However, Lemma \ref{lem:ga}
implies $\widetilde{Q}\bigl(\x \tau_{1} \cdots
\tau_{a}e(\beta,i)\bigr)
=\widetilde{Q}\bigl(e(\beta,i)\bigr)\x[2]\tau_{2} \cdots \tau_{a+1}
=0$. \QED

\begin{Thm} \label{thm:Q}
For each $\nu \in I^{\beta}$,  set
$$A_{\nu} = a_{i}^\Lambda(x_{1})\prod_{1 \le a
\le n, \nu_{a} \neq i} Q_{i, \nu_{a}}(x_{1}, x_{a+1}).$$  Then the
following diagram is commutative, in which the vertical arrow is the
multiplication by $A_\nu$ from the right.
\eq\label{eq:Qcomm}&&
\ba{l}
\xymatrix@C=17ex@R=6ex{\dfrac{R(\beta+\alpha_i) e(i, \nu)} {R(\beta + \alpha_i)
\x[2] R^{1}(\beta) e(i, \nu)}
 \ar[d]^-{A_{\nu}} \ar[r]^-{P=\x \tau_{1} \cdots \tau_{n}}
&  \dfrac{R(\beta+\alpha_i) e(\nu, i)} {R(\beta+\alpha_i)
\x R(\beta) e(\nu, i)} \ar[dl]^-{\quad Q = g_{n} \cdots
g_{1}}
\\ \dfrac{R(\beta+\alpha_i) e(i, \nu)} {R(\beta + \alpha_i)
\x[2] R^{1}(\beta) e(i, \nu)} &
 }\ea
\eneq
\end{Thm}

\begin{proof}
We will verify
\eq
&&\ba{rcl}
\x \tau_{1} \cdots \tau_{n} g_{n} \cdots g_{1}e(i,\nu)
&=&\x \tau_{1} \cdots \tau_{n} e(\nu,i)g_{n}\cdots g_{1}\\[1.5ex]
&\equiv& A_{\nu}e(i,\nu)
\mod{R(\beta + \alpha_i)\x[2] R^{1}(\beta) e(i, \nu)}
\ea\label{eq:taug}
\eneq
by induction on $n$. In the course of the proof, we use the fact
 that
\eq
&&
\text{$R(\beta + \alpha_i)\x[2] R^{1}(\beta) e(i, \beta)$
is a right $\cor[x_1,\ldots, x_{n+1}]$-module.}
\label{eq:24}
\eneq

Note that
\begin{equation}
\tau_{n} e(\nu,i) g_{n} = \left\{
\ba{ll}\tau_{n} e(\nu, i) \tau_{n}
= Q_{i, \nu_{n}}(x_{n}, x_{n+1}) e(\nu', i, \nu_{n}) & \text{if $\nu_{n} \neq i$,}
 \\[2.5ex]
\tau_{n}\bigl((x_{n+1} - x_{n}) - (x_{n+1}-x_{n})^2 \tau_{n}\bigr)e(\nu,i) \\
\hs{20ex}=\tau_{n}(x_{n+1} - x_{n}) e(\nu,i)& \text{if $\nu_{n} = i$,}
\ea
\right.\label{eq:tngn}
\end{equation}
where $\nu' = (\nu_1, \cdots, \nu_{n-1})$.

\medskip
\noindent Assume first $n=1$. If $\nu_1\not=i$, then \eqref{eq:taug}
is obvious by \eqref{eq:tngn}. If $\nu_1=i$, then  $\x\tau_{1}
e(\nu,i) g_{1} =\x\tau_{1}(x_{2} - x_{1}) e(i,i)$. Then by Lemma
\ref{lem:ga},
it is equal to
\begin{equation*}
\begin{aligned}
& \x(\tau_{1}(x_{2} - x_{1})-1) e(i,i)+\x e(i,i)
\\
& =(\tau_{1}(x_{2} - x_{1})-1)\x[2] e(i,i)+\x e(i,i)
\\
& \equiv A_\nu e(i,i)
\mod{R(\beta + \alpha_i)\x[2] R^{1}(\beta) e(i, i)}.
\end{aligned}
\end{equation*}

\medskip
Thus we may assume that $n>1$.

\medskip

\noindent
(i)\ First assume $\nu_{n} \neq i$. Then we have
\begin{equation*}
\begin{aligned}
 \x \tau_{1}  \cdots \tau_{n}  g_{n} \cdots g_{1} e(i,
\nu) &= \x \tau_{1} \cdots \tau_{n-1} Q_{i,
\nu_{n}}(x_{n}, x_{n+1}) g_{n-1} \cdots g_{1} e(i, \nu) \\
& =\x \tau_{1} \cdots \tau_{n-1} g_{n-1} \cdots g_{1}
 e(i, \nu)Q_{i, \nu_{n}}(x_{1}, x_{n+1}).
\end{aligned}
\end{equation*}
By the induction hypothesis, we have
$$\x \tau_{1} \cdots \tau_{n-1} g_{n-1} \cdots g_{1}
 e(i, \nu)\equiv A_{\nu'}e(i,\nu) \
\mod{R(\beta + \alpha_i)\x[2] R^{1}(\beta) e(i, \nu)}.$$
Hence \eqref{eq:24} implies that
\eqn
&&\x \tau_{1} \cdots \tau_{n-1} g_{n-1} \cdots g_{1} e(i, \nu)
Q_{i, \nu_{n}}(x_{1}, x_{n+1})\\
&&\hs{15ex}\equiv A_{\nu'}e(i,\nu)Q_{i, \nu_{n}}(x_{1}, x_{n+1})
\mod{R(\beta + \alpha_i)\x[2] R^{1}(\beta) e(i, \nu)}.
\eneqn
Then the desired assertion
follows from
$A_{\nu'}Q_{i, \nu_{n}}(x_{1}, x_{n+1})e(i,\nu)= A_{\nu}e(i,\nu)$.

\noindent
(ii)\ Now assume $\nu_{n} =i$. Set $\nu'' = (\nu_{1}, \ldots, \nu_{n-2})$.
Then we have
\begin{equation} \label{eq:38}
\begin{aligned}
& \x \tau_{1}  \cdots \tau_{n}  g_{n} \cdots
g_{1}e(i, \nu) =  \x \tau_{1} \cdots \tau_{n}
(x_{n+1}-x_{n}) g_{n-1} \cdots g_{1} e(i, \nu)\\
&\hs{20ex} =  \x \tau_{1} \cdots \tau_{n} e(\nu'',
\nu_{n-1}, i,i ) g_{n-1} \cdots g_{1} (x_{n+1} - x_{1}).
\end{aligned}
\end{equation}

\noindent
(a)\ If $\nu_{n-1} \neq i$, then
\eqn
\tau_{n-1} \tau_{n} g_{n-1} e(\nu'', i, \nu_{n-1}, i) &=& \tau_{n-1}
\tau_{n} \tau_{n-1} e(\nu'', i, \nu_{n-1}, i) \\
&=& (\tau_{n} \tau_{n-1} \tau_{n} -
\overline{Q}_{n-1, n, n+1}) e(\nu'', i, \nu_{n-1}, i).
\eneqn
(See \eqref{eq:tau3}.)  Hence
\begin{equation}
\begin{aligned}
& \x  \tau_{1} \cdots \tau_{n} g_{n} \cdots g_{1}e(i, \nu)  \\
& =\x\tau_{1} \cdots \tau_{n-2} (\tau_{n}
\tau_{n-1} \tau_{n} - \overline{Q}_{n-1,n,n+1}) e(\nu'', i,
\nu_{n-1}, i) g_{n-2} \cdots g_{1} (x_{n+1} - x_{1}) \\
& = \tau_{n}\x \tau_{1} \cdots
\tau_{n-1} e(\nu'', i,i, \nu_{n-1}) g_{n-2} \cdots g_{1} \tau_{n}
(x_{n+1} - x_{1}) \\
& \ \ \ -  \x \tau_{1} \cdots \tau_{n-2}
g_{n-2} \cdots g_{1} \overline{Q}_{1,n,n+1}(x_{n+1} - x_{1})e(i, \nu).
\end{aligned}\label{eq:qtg}
\end{equation}
The first term is equal to
\begin{equation*}
\begin{aligned}
&\tau_{n}\x \tau_{1} \cdots
\tau_{n-1} e(\nu'', i,i, \nu_{n-1}) g_{n-2} \cdots g_{1} \tau_{n}
(x_{n+1} - x_{1}) \\
& =\tau_{n} \x \tau_{1} \cdots
\tau_{n-1} g_{n-2} \cdots g_{1} (x_{n}-x_{1}) \tau_{n} e(i,\nu)\\
&= \tau_{n} \Bigl(\x \tau_{1} \cdots
\tau_{n-1} g_{n-1} \cdots g_{1}e(i,s_{n-1}\nu)\Bigr) \tau_{n}
\quad\text{by \eqref{eq:38}}
 \\
&\equiv \tau_{n}  A_{(\nu'', i)}e(i, s_{n-1} \nu) \tau_{n}
\mod{R(\beta + \alpha_i)\x[2] R^{1}(\beta) e(i, \nu)}\\
&= A_{(\nu'',i)} Q_{\nu_{n-1}, \nu_{n}}(x_{n}, x_{n+1})e(i,\nu)\\
&= A_{\nu''} Q_{\nu_{n-1},i}(x_{n}, x_{n+1})e(i,\nu).
\end{aligned}
\end{equation*}
On the other hand, the second term in \eqref{eq:qtg} is equivalent to
\begin{equation*}
\begin{aligned}
& -A_{\nu''}(Q_{i, \nu_{n-1}}(x_{n+1}, x_{n}) - Q_{i,\nu_{n-1}}(x_{1}, x_{n}))
e(i, \nu) \\
&\hs{10ex} = -A_{\nu''} Q_{\nu_{n-1},i} (x_{n}, x_{n+1})e(i, \nu) +
A_{\nu''} Q_{i, \nu_{n-1}}(x_{1}, x_{n})e(i, \nu)\\
&\hs{25ex} \mod{R(\beta + \alpha_i)\x[2] R^{1}(\beta) e(i, \nu)}.
\end{aligned}
\end{equation*}
Hence these two terms add up to
$$A_{\nu''} Q_{i, \nu_{n-1}}(x_{1}, x_{n})e(i, \nu) = A_{\nu}e(i, \nu).$$

\medskip
\noindent
(b)\ Finally, suppose $\nu_{n} = \nu_{n-1}=i$. In this
case, we have
\begin{equation*}
 \x\tau_{1} \cdots \tau_{n} g_{n} \cdots
g_{1}e(i, \nu)  =  \x \tau_{1} \cdots \tau_{n} g_{n-1}
\cdots g_{1} (x_{n+1}-x_{1})e(i, \nu)
\end{equation*}
and
\begin{equation*}
\begin{aligned}
\tau_{n-1} \tau_{n} & g_{n-1}e(\nu'',i,i,i)  = \tau_{n-1} \tau_{n}
\bigl((x_{n}-x_{n-1})- \tau_{n-1}(x_{n}-x_{n-1})^2)\bigr)e(\nu'',i,i,i) \\
& =\tau_{n-1} \tau_{n} (x_{n}-x_{n-1})e(\nu'',i,i,i) - \tau_{n}
\tau_{n-1} \tau_{n} (x_{n} - x_{n-1})^2e(\nu'',i,i,i).
\end{aligned}
\end{equation*}
Hence
\begin{equation*}
\begin{aligned}
& \x \tau_{1} \cdots \tau_{n} g_{n} \cdots
g_{1}e(i, \nu)  \\
& =  \x \tau_{1} \cdots \tau_{n} (x_{n}-x_{n-1})
g_{n-2} \cdots g_{1} (x_{n+1} -x_{1})e(i, \nu) \\
&\ \  - \ x \tau_{1} \cdots \tau_{n-2}
(\tau_{n} \tau_{n-1} \tau_{n}) (x_{n}-x_{n-1})^2 g_{n-2} \cdots
g_{1} (x_{n+1} - x_{1})e(i, \nu) \\
& =  \x \tau_{1} \cdots \tau_{n-1} g_{n-2}
\cdots g_{1} \tau_{n} (x_{n} -x_{1}) (x_{n+1} -x_{1}) e(i, \nu)\\
& \ \ - \tau_{n} x_{1}^{\lambda} \tau_{1} \cdots \tau_{n-1}
g_{n-2} \cdots g_{1} \tau_{n} (x_{n}-x_{1})^2  (x_{n+1} -x_{1}) e(i, \nu) \\
& = \x \tau_{1} \cdots \tau_{n-1} g_{n-2}
\cdots g_{1} (x_{n}-x_{1}) (x_{n+1} -x_{1}) \tau_{n}  e(i, \nu)\\
& \ \ -  \tau_{n} \x\tau_{1} \cdots \tau_{n-1} g_{n-2} \cdots g_{1}
(x_{n} -x_{1}) (x_{n+1} -x_{1}) \tau_{n} (x_{n}
- x_{1}) e(i, \nu)\\
& =  \x \tau_{1} \cdots \tau_{n-1} (x_{n}
-x_{n-1}) g_{n-2} \cdots g_{1} (x_{n+1} -x_{1}) \tau_{n}e(i, \nu) \\
& \ \ - \tau_{n} \x \tau_{1} \cdots \tau_{n-1} (x_{n} - x_{n-1})
g_{n-2} \cdots g_{1} (x_{n+1} - x_{1}) \tau_{n}
(x_{n} -x_{1}) e(i, \nu) \\
& \equiv A_{\nu'} (x_{n+1} -x_{n}) \tau_{n}  e(i, \nu)- \tau_{n} A_{\nu'} (x_{n+1}
- x_{1}) \tau_{n} (x_{n} - x_{1}) e(i, \nu) \\
& = A_{\nu'} \Bigl((x_{n+1} - x_{1}) \tau_{n} - \tau_{n} (x_{n+1} -
x_{1}) \tau_{n} (x_{n} - x_{1})\Bigr) e(i,\nu).
\end{aligned}
\end{equation*}
Note that
$$ \tau_{n} (x_{n+1} - x_{1})e(i,\nu)
=\bigl( (x_{n} - x_{1}) \tau_{n} + 1 \bigr)e(i,\nu)\quad
\text{and} \quad \tau_{n}^2e(i,\nu)  =0,$$ which implies $$\tau_{n} (x_{n+1}
- x_{1} ) \tau_{n}e(i,\nu) = \tau_{n}e(i,\nu),$$ and hence
$$\Bigl((x_{n+1} - x_{1}) \tau_{n} - \tau_{n} (x_{n+1} - x_{1}) \tau_{n}
(x_{n} - x_{1})\Bigr)e(i,\nu) =e(i,\nu).$$ Since $\nu = (\nu', i)$, we have $A_{\nu'} =
A_{\nu}$, which completes the proof.
\end{proof}

Note that $K_1$ is a projective right $\bl
R^\Lambda(\beta)\otimes\cor[t_i]\br$-module by
Lemma~\ref{lem:Kproj}, and the right multiplication by $\sum_{\nu\in
I^\beta} A_\nu e(i,\nu)\in R(\beta+\alpha_i)$ on $K_1$ is equal to
the action of $\sum_{\nu\in I^\beta} a_i^\Lambda(t_i) \prod_{1 \le a
\le n, \nu_{a} \neq i} Q_{i, \nu_{a}}(t_i, x_{a}) e(\nu)\in
R^\Lambda(\beta)\otimes\cor[t_i]$, which is a monic polynomial (up
to an invertible multiple). Hence $Q\circ P$ is injective by the
following elementary lemma whose proof is omitted.
\begin{Lem}\label{lem:inj}
Let $A$ be a  ring and $K$ a projective $A\otimes_{\Z}
\Z[t]$-module, where $t$ is an indeterminate. If $f=\sum_{0\le k\le
m}a_kt^k$ is an element of $A\otimes_{\Z} \Z[t]$ such that $a_m$ is
an invertible element of $A$, then the multiplication by $f$ on $K$
is an injective endomorphism of $K$.
\end{Lem}

Hence we obtain the following lemma.

\begin{Lem} \label{lem:resolution}
The homomorphism $P\cl K_{1} \longrightarrow K_{0} $ is injective, and
\begin{equation} \label{eq:exact}
0 \longrightarrow K_{1} \overset{P} \longrightarrow K_{0}
\overset{\pi} \longrightarrow F^{\Lambda} \longrightarrow 0
\end{equation}
is an exact sequence of $(R(\beta+\alpha_i),
R^{\Lambda}(\beta)\otimes\cor[t_i])$-bimodules.
\end{Lem}

Recall that the action of $t_i$ is the right multiplication by $x_1$
(resp.\ $x_{n+1}$) on $K_1$ (resp.\ on $K_0$ and $F^\Lambda$). Since
both $K_{1}$ and $K_{0}$ are projective right
$R^{\Lambda}(\beta)[t_i]$-modules  by Lemma~\ref{lem:Kproj},
\eqref{eq:exact} is a projective resolution of $F^{\Lambda}$ as a
right $R^{\Lambda}(\beta)[t_i]$-module.

By Lemma~\ref{lem:integrable}, there exists a monic polynomial
$g(u)$ such that  $g(x_{n+1})=0$ as an element of
$R^\Lambda(\beta+\alpha_i)$.  Hence Theorem~\ref{th:proj} follows
from Lemma~\ref{lem:pro} below, and Theorem~\ref{th:main} is its
consequence. Note that the homomorphism $P$ is homogeneous of degree
$(\alpha_i\mid2\Lambda-\beta)$ and it induces the morphism
$q^{(\alpha_i|2\Lambda-\beta)}\bF_iM\To F_iM$ in
Theorem~\ref{th:main}.

\begin{Lem}\label{lem:pro}
Let $A$ be a ring and let $A[t]=A\otimes_{\Z} \Z[t]$ be the
polynomial ring in $t$ with coefficients in $A$. Let $a(t)$ be a
monic polynomial in $t$ with coefficients in the center $Z(A)$ of
$A$ and let $M$ be an $A[t]$-module such that $a(t)M=0$. If $M$ has
a projective dimension at most $1$ as an $A[t]$-module, that is, if
there exists an exact sequence of $A[t]$-modules \eq &&0\To
P_1\To[\;j\;] P_0\To M\To0\label{eq:resol} \eneq with projective
$A[t]$-modules $P_0$ and $P_1$, then $M$ is a projective $A$-module.
\end{Lem}
\Proof
Let us take an exact sequence as in \eqref{eq:resol},
and let $a_i\in\End_{A[t]}(P_i)$ ($i=0,1$) be the endomorphism
of $P_i$ obtained by the multiplication of $a(t)$.
By Lemma~\ref{lem:inj} we have an exact sequence
\eq
&&0\to P_i\To[a_i]P_i\to P_i/a_iP_i\to0.
\label{eq:extseq}
\eneq
On the other hand, we have a commutative diagram with exact rows:
\eqn
&&\xymatrix{
0\ar[r]&P_1\ar[r]^j\ar[d]^{a_1}&P_0\ar@{.>}[ld]|h\ar[d]^{a_0}
\ar[r]&M\ar[d]^{a(t)=0}\ar[r]&0
\\
0\ar[r]&P_1\ar[r]_j\ar@{.>}@/^1.5pc/[u]|p& P_0 \ar[r]&M\ar[r]&0 }
\eneqn Since $a(t)\vert_M=0$, there exists an $A[t]$-linear
homomorphism $h$ such that $a_0=j\circ h$. Hence, we have $j\circ
(h\circ j)= a_0\circ j=j\circ a_1$. Since $j$ is a monomorphism, we
obtain $h\circ j=a_1$.

Since $A[t]/a_iA[t]$ is a projective $A$-module and
$P_i$ is a projective $A[t]$-module,
$P_i/a_iP_i\simeq (A[t]/a_iA[t])\otimes_{A[t]}P_i$ is a projective $A$-module.
Hence \eqref{eq:extseq} splits as an exact sequence of $A$-modules.
Hence there exists $p\in \End_A(P_1)$ such that $p\circ a_1=\id_{P_1}$.
Hence $p\circ h\circ j=\id_{P_1}$.
Therefore the exact sequence \eqref{eq:resol} splits,
and $M$ is a projective $A$-module.
\QED

\bigskip
We need the following lemma later.
\begin{Lem}\label{lem:ABP}
Set
\eqn
&&\mathsf{A}=\sum_{\nu\in I^\beta}A_\nu e(i,\nu),\\
&&\mathsf{B}=\sum_{\nu\in I^\beta}
 a_i^\Lambda(x_{n+1}) \prod_{\nu_{a}\neq i} Q_{\nu_{a}, i}(x_{a}, x_{n+1})e(\nu,i).\eneqn
Then we have a commutative diagram
\eqn &&\xymatrix@C=7ex{
K_1\ar[r]^P\ar[d]_{\mathsf{A}}
&K_0\ar[d]^{\mathsf{B}}\ar[dl]|Q\\
K_1\ar[r]_P&K_0\,. } \eneqn
Here the vertical arrows are the right
multiplications by $\mathsf{A}$ and $\mathsf{B}$, respectively.
\end{Lem}
\Proof
We know already the commutativity of the upper triangle
and the square.
Let us show the commutativity of the lower triangle.
We have $P\circ Q\circ P=P\circ \mathsf{A}=\mathsf{B}\circ P$.
Hence $(P\circ Q-\mathsf{B})\circ P=0$.
By the exact sequence \eqref{eq:exact},
$P\circ Q-\mathsf{B}\in \End_{R^\Lambda(\beta)[t_i]}(K_0)$
factors through $F^\Lambda$.
Since $F^\Lambda$ is annihilated by the right multiplication by
a monic polynomial $g(t_i)$,
the image of $P\circ Q-\mathsf{B}$
is also  annihilated by $g(x_{n+1})$.
Since $g(x_{n+1})\cl K_0\to K_0$ is a monomorphism, we obtain
$P\circ Q-\mathsf{B}=0$.
\QED

\begin{Rem}\label{rem:proj}\bnum
\item
Both $\sum\limits_{\nu\in I^\beta}\prod\limits_{\nu_a=i}(t_i-x_a)e(\nu)$
and $\sum\limits_{\nu\in I^\beta}\prod\limits_{\nu_a\not=i}Q_{i,\nu_a}(t_i,x_a)e(\nu)$
belong to the center of $R(\beta)\otimes\cor[t_i]$
(cf.\ \cite[Theorem 2.9]{KL09}).
\item
For any $n\ge m$, $R^\Lambda(n)$ is a projective
$R^\Lambda(m)$-module. Indeed, $R^\Lambda(m+1)$ is a projective
$R^\Lambda(m)$-module by Theorem~\ref{th:proj}, and the general case
follows by induction on $n$. In particular, $R^\Lambda(n)$ is a
projective $\cor$-module. \label{remark1}
\enum
\end{Rem}

\vskip 3em


\section{$\gsl_2$-categorification} \label{sec:V(Lambda)}

In this section, we will show that the functors $E_i^\Lambda$ and
$F_i^\Lambda$ satisfy certain commutation relations similar to those
between generators of the Lie algebra $\gsl_2$.

\subsection{Commutation relations between $E_{i}^{\Lambda}$
and  $F_i^\Lambda$} For the adjoint pair $(F_{i}^{\Lambda},
E_{i}^{\Lambda})$, consider the adjunction transformations
$\varepsilon\cl F_{i}^{\Lambda} E_{i}^{\Lambda} \rightarrow \Id$ and
$\eta\cl \Id \rightarrow E_{i}^{\Lambda} F_{i}^{\Lambda}$, and the
natural transformations $x_{E_{i}^{\Lambda}}$, $x_{F_{i}^{\Lambda}}$
as in Section \ref{sec:R}.  For example,
$x_{F_{i}^{\Lambda_{i}}}$ is given by the right multiplication by
$x_{n+1}$ on $R^{\Lambda}(\beta+\alpha_i)e(\beta, i)$.  Thus we
obtain  the following commutative diagram:

\begin{equation} \label{B:comm}
\ba{c}
\xymatrix{ \Hom_{R^{\Lambda}(\beta+\alpha_i)}(F_{i}^{\Lambda}(M), N)
\ar[d]^-{x_{F_i^{\Lambda}}}\ar[r]^-{\sim} &
\Hom_{R^{\Lambda}(\beta)}(M, E_{i}^{\Lambda}(N))
\ar[d]^-{x_{E_{i}^{\Lambda}}}
\\ \Hom_{R^{\Lambda}(\beta+ \alpha_i)} (F_{i}^{\Lambda}(M), N)
 \ar[r]^-{\sim} & \Hom_{R^{\Lambda}(\beta)}(M, E_{i}^{\Lambda}(N))\,.
 }\ea
\end{equation}

\begin{Thm} \label{thm:A}
For $i \neq j$, there exists a natural isomorphism
$$
q^{-(\alpha_i|\alpha_j)}F_{j}^{\Lambda} E_{i}^{\Lambda}
\isoto E_{i}^{\Lambda}F_{j}^{\Lambda}.$$
\end{Thm}

\begin{proof}
By Corollary \ref{cor:eRe}, there is an isomorphism
$$e(n,i) R(n+1) e(n,j) \overset{\sim} \longrightarrow R(n) e(n-1, j)
\otimes_{R(n-1)} e(n-1, i) R(n).$$
Applying the functor $R^{\Lambda}(n)\otimes_{R(n)}
\scbul \otimes_{R(n)} R^{\Lambda}(n)$, the
right-hand side yields
$$R^{\Lambda}(n) e(n-1, j) \otimes_{R^{\Lambda}(n-1)} e(n-1, i)
R^{\Lambda}(n) = F_{j}^{\Lambda} E_{i}^{\Lambda}R^{\Lambda}(n),$$ and the
left-hand side is equal to
$$\dfrac{ e(n,i) R(n+1) e(n,j)} {e(n,i) R(n) \x R(n+1)
e(n,j) + e(n,i) R(n+1) \x R(n) e(n,j)}.$$ Since
\begin{equation*}
\begin{aligned}
E_{i}^{\Lambda} F_{j}^{\Lambda}R^{\Lambda}(n) &= e(n,i)R^{\Lambda}(n+1)e(n,j)
\otimes_{R(n)} R^{\Lambda}(n) \\
& =\dfrac{e(n,i)R(n+1)e(n,j)}{e(n,i) R(n+1) \x R(n+1)
e(n,j)},
\end{aligned}
\end{equation*}
it suffices to show that
\begin{equation}
\begin{aligned}
& e(n,i) R(n+1)\x R(n+1) e(n,j) \\
& = e(n,i)R(n)\x R(n+1) e(n,j) + e(n,i) R(n+1)
\x R(n) e(n,j).
\end{aligned}
\label{eq:cyclorel}
\end{equation}
Indeed, we have
\begin{equation*}
\begin{aligned}
& R(n+1) \x R(n+1) = \sum_{a=1}^{n+1} R(n+1)
\x \tau_{a} \cdots \tau_{n} R(n,1) \\
& = R(n+1)\x R(n,1) + R(n+1) \x \tau_{1}
\cdots \tau_{n} R(n,1) \\
&= R(n+1) \x R(n) + \sum_{a=1}^{n+1} R(n,1) \tau_{n}
\cdots \tau_{a} \x \tau_{1} \cdots \tau_{n} R(n,1)\\
&=R(n+1) \x R(n) + R(n,1) \x R(n+1) +
R(n,1) \tau_{n} \cdots \tau_{1} \x \tau_{1} \cdots
\tau_{n} R(n,1) \\
& = R(n+1) \x R(n) + R(n) \x R(n+1) +
R(n,1) \tau_{n} \cdots \tau_{1} \x \tau_{1} \cdots
\tau_{n} R(n,1).
\end{aligned}
\end{equation*}
Since $i \neq j$, we have
\begin{equation*}
\begin{aligned}
& e(n,i) R(n,1) \tau_{n} \cdots \tau_{1} \x \tau_{1}
\cdots \tau_{n} R(n,1) e(n,j) \\
&=R(n,1) \tau_{n} \cdots \tau_{1} e(i,n) \x e(j,n)
\tau_{1} \cdots \tau_{n} R(n,1) =0,
\end{aligned}
\end{equation*}
which proves our assertion \eqref{eq:cyclorel}.
\end{proof}

\noindent
The natural transformation $q_{i}^{-2} F_{i} E_{i} \to E_{i}
F_{i}$ defined in Proposition~\ref{prop:R(n)xR(n)}
induces a natural transformation $q_{i}^{-2} F_{i}^{\Lambda}
E_{i}^{\Lambda} \to E_{i}^{\Lambda} F_{i}^{\Lambda}$.
Moreover, there exists a natural transformation
$$q_{i}^{2k} \Id\longrightarrow E_{i}^{\Lambda} F_{i}^{\Lambda} \quad (k\ge 0)$$
given by the following commutative diagram:
\begin{equation} \label{C:comm}
\ba{c} \xymatrix@C=15ex{\Id  \ar[d]_{ \eta }\ar[r]^-{\eta} &
E_{i}^{\Lambda} F_{i}^{\Lambda}  \ar[d]^-{E_{i}^{\Lambda} \circ \,
(x_{F_{i}^{\Lambda}})^{k}}
\\ E_{i}^{\Lambda} F_{i}^{\Lambda}
 \ar[r]_{ (x_{E_{i}^{\Lambda}})^k \circ \, F_{i}^{\Lambda}} &
 q_i^{-2k} E_{i}^{\Lambda} F_{i}^{\Lambda}\,.
 }\ea
\end{equation}
Note that the commutativity of \eqref{C:comm} follows from
\eqref{B:comm}. Similarly, there exists a natural transformation
$$F_{i}^{\Lambda} E_{i}^{\Lambda} \longrightarrow q_{i}^{-2k} \Id\quad(k \ge 0)$$
given by the following commutative diagram:
\begin{equation} \label{D:comm}
\ba{c} \xymatrix@C=15ex{q_i^{2k} F_{i}^{\Lambda} E_{i}^{\Lambda}
\ar[d]_{ F_{i}^{\Lambda} \circ \, (x_{E_{i}^{\Lambda}})^{k}
}\ar[r]^-{ (x_{F_{i}^{\Lambda}})^{k} \circ \, E_{i}^{\Lambda}} &
F_{i}^{\Lambda} E_{i}^{\Lambda} \ar[d]^-{\varepsilon}
\\ F_{i}^{\Lambda} E_{i}^{\Lambda}
 \ar[r]_{ \varepsilon } & \Id\,.
  }\ea
\end{equation}

Now we can state another main theorem of our paper.

\begin{Thm} \label{thm:M}
Let $\lambda = \Lambda - \beta$. Then there exist natural
isomorphisms of endofunctors on $\Mod(R^{\Lambda}(\beta))$ given
below.
\bna
\item If $\langle h_{i}, \lambda \rangle \ge 0$, then we have an isomorphism
$$q_{i}^{-2} F_{i}^{\Lambda} E_{i}^{\Lambda} \oplus
\bigoplus_{k=0}^{\langle h_{i}, \lambda \rangle -1} q_{i}^{2k} \Id
\isoto E_{i}^{\Lambda} F_{i}^{\Lambda}.$$
\item If $\langle h_{i}, \lambda \rangle \le 0$, then we have an isomorphism
$$q_{i}^{-2} F_{i}^{\Lambda} E_{i}^{\Lambda} \isoto E_{i}^{\Lambda} F_{i}^{\Lambda} \oplus
\bigoplus_{k=0}^{-\langle h_{i}, \lambda \rangle -1} q_{i}^{-2k-2}
\Id.$$
\end{enumerate}
\end{Thm}
Note that in  \cite[{\S 4.1.3}]{R08} it is one of the axioms
for the categorification (see also \cite[Theorem 5.27]{CR08}).

\subsection{Proof of Theorem~\ref{thm:M}}
In order to prove this theorem, we consider the following
commutative diagrams with exact rows and columns for $M \in
\Mod(R^{\Lambda}(\beta))$ (see Theorem~\ref{thm:E_i},
Theorem~\ref{thm:bF} and Theorem~\ref{th:main}):
\begin{equation} \label{E:comm}\ba{c}
\xymatrix{
{} & 0 \ar[d] & 0 \ar[d]  & q_{i}^{-2}M & {} \\
0 \ar[r] & q^{(\alpha_i | 2 \Lambda - \beta)} \bF_{i} E_{i}
M \ar[d] \ar[r] & q_{i}^{-2} F_{i} E_{i} M \ar[d] \ar[r]
\ar[ur]^{\varepsilon} & q_{i}^{-2} F_{i}^{\Lambda}
E_{i}^{\Lambda} M  \ar[d] \ar[r] & 0 \\
0 \ar[r] & q^{(\alpha_{i}| 2 \Lambda - \beta)} E_{i}
\bF_{i} M \ar[d] \ar[r]  &
E_{i} F_{i} M \ar[d] \ar[r] & E_{i}^{\Lambda} F_{i}^{\Lambda} M  \ar[r] & 0 .\\
{} & q^{(\alpha_{i} | 2 \Lambda - 2\beta)} \cor[t_{i}] \otimes M
\ar[d] \ar[r] &
\cor[t_{i}] \otimes M \ar[d] & {} & {} \\
{} & 0 & 0 & {} & {}
  }\ea
\end{equation}

At the kernel level, the commutative diagram \eqref{E:comm}
corresponds to the following commutative diagram of $(R(\beta),
R^\Lambda(\beta))$-bimodules
\begin{equation} \label{F:comm}
\ba{c}
\xymatrix@C=3.5ex{
{} & 0 \ar[d] & 0 \ar[d]  & q_{i}^{-2} R^{\Lambda}(\beta) & {} \\
0 \ar[r] &  q^{(\alpha_i | 2 \Lambda - \beta)} K_{1}'
\ar[d] \ar[r]^{P'}  & q_{i}^{-2} K_{0}'
\ar[d]^{F=\tau_{n}}
 \ar[r]_-{G } \ar[ur]^{E}
&q_i^{-2} F_{i}^{\Lambda} E_{i}^{\Lambda} R^{\Lambda}(\beta)  \ar[d] \ar[r] & 0 \\
 0 \ar[r] & q^{(\alpha_{i}| 2 \Lambda -
\beta)}E_iK_{1}  \ar[d]_{B} \ar[r]^{P} &
E_iK_{0} \ar[d]^{C} \ar[r] &E_i F^{\Lambda}\ar[r] & 0 \\
{} & q^{(\alpha_{i} | 2 \Lambda - 2\beta)} \cor[t_{i}] \otimes
R^{\Lambda}(\beta) \ar[d] \ar[r]^-{A} &
\cor[t_{i}] \otimes R^{\Lambda}(\beta)  \ar[d] & {} & {} \\
{} & 0 & 0 & {} & {}
  }\ea
\end{equation}
Here, we have
\begin{equation} \label{eq:K'}
\begin{aligned}
 K_{0}' &\seteq F_iE_iR^\Lambda(\beta)= R(\beta) e(\beta-\alpha_{i}, i)
\otimes_{R(\beta-\alpha_{i})}e(\beta-\alpha_i,i) R^{\Lambda}(\beta), \\
K_{1}' &\seteq \bF_iE_iR^\Lambda(\beta)
=R(\beta) e(i, \beta-\alpha_{i})\otimes_{R(\beta-\alpha_{i})}
e(\beta-\alpha_i,i) R(\beta)\otimes_{R(\beta)}R^\Lambda(\beta)\\
&=R(\beta) e(i, \beta-\alpha_{i})\otimes_{R(\beta-\alpha_{i})}
e(\beta-\alpha_i,i)R^{\Lambda}(\beta),
\end{aligned}
\end{equation}
and
\begin{equation} \label{eq:E_iK}
\begin{aligned}
 F_{i}^{\Lambda} E_{i}^{\Lambda} R^{\Lambda}(\beta)&
= R^\Lambda(\beta) e(\beta-\alpha_{i}, i)
\otimes_{R(\beta-\alpha_{i})}e(\beta-\alpha_i,i) R^{\Lambda}(\beta),\\[.5ex]
 E_iF^{\Lambda}& =E_i^\Lambda F_i^\Lambda R^{\Lambda}(\beta)
=e(\beta, i)R^{\Lambda}(\beta+\alpha_i) e(\beta, i) \\
&= \dfrac{e(\beta, i)R(\beta+\alpha_i) e(\beta, i)}{e(\beta,
i)R(\beta+\alpha_i)
\x R(\beta+\alpha_i) e(\beta, i)}, \\[.5ex]
 E_iK_{0}&= E_i F_i R^{\Lambda}(\beta)
=e(\beta,i)R(\beta+ \alpha_i) e(\beta, i) \tens\limits_{R(\beta)}
R^{\Lambda}(\beta) \\
& = \dfrac{e(\beta,i)R(\beta+\alpha_i) e(\beta,i)}%
{e(\beta,i)R(\beta+\alpha_i)\x R(\beta) e(\beta, i)}, \\[.5ex]
 E_iK_{1}& =  E_i \bF_i R^{\Lambda}(\beta)
 =e(\beta,i)R(\beta+\alpha_i) e(i, \beta) \tens\limits_{R(\beta)}
R^{\Lambda}(\beta)\\ &= \dfrac{e(\beta,i)R(\beta+\alpha_i) e(i,\beta)}%
{e(\beta,i)R(\beta+\alpha_i) \x[2] R^{1}(\beta) e(i,
\beta)}.
\end{aligned}
\end{equation}
We also note the following properties:
\be[(1)]
\item The map $P$ is the right multiplication by $\x
\tau_{1} \cdots \tau_{n}$  given in Section \ref{sec:RLambda}, and is
$\bl R(\beta), R^{\Lambda}(\beta) \otimes \cor[t_{i}]\br$-bilinear.

\item
Similarly, the map $P'$ is given by the right multiplication by
$\x\tau_1\cdots\tau_{n-1}$ on $R(\beta) e(i, \beta-\alpha_{i})$.

\item  The map $E$ is given by $x\otimes y\mapsto xy$ 
and the map $F$ is given by
$x\otimes y\mapsto x\tau_ny$
(see Proposition \ref{prop:R(n)xR(n)}).

\item The map $C$ is the cokernel map of $F$. Hence $C$ is $(R(\beta),
R^{\Lambda}(\beta))$-bilinear but does {\it not} commute with
$t_{i}$.

\item The map $B$ is written as $\varphi$ in Corollary \ref{cor:varphi-1}.
Thus it is given by taking the coefficient of $\tau_{n} \cdots
\tau_{1}$ and is $(R(\beta) \otimes \cor[x_{n+1}], \cor[x_{1}]
\otimes R^\Lambda(\beta))$-bilinear.

\item  The map $A$ is defined by chasing the
diagram. It is $R^{\Lambda}(\beta)$-bilinear but {\it not}
$\cor[t_{i}]$-linear.
\item The map $G$ is the canonical projection, which is
$\bl R(\beta)\otimes\cor[x_{n+1}],R^\Lambda(\beta)\otimes\cor[x_{n+1}]\br$-bilinear.
\ee

\medskip
We write $p$ for the number of times $\alpha_{i}$ appears in
$\beta$. Define an invertible element $\gamma \in \cor^{\times}$ by
\eq\label{eq:gamma}&&\hs{.5ex}
(-1)^{p} \prod_{\nu_{a} \neq i} Q_{i, \nu_{a}}(t_{i}, x_{a}) =
\gamma^{-1} t_{i}^{-\langle h_{i}, \beta \rangle + 2p} +
\Bigl(\text{terms of degree $< - \langle h_{i}, \beta \rangle + 2p$
in $t_i$}\Bigr).
\eneq
For $\lambda=\Lambda-\beta$, define
\begin{equation} \label{eq:phi}
\varphi_{k} = A(t_{i}^k) \in
\cor[t_{i}] \otimes R^{\Lambda}(\beta),
\end{equation}
which is of
degree $2(\alpha_{i}| \lambda) + k(\alpha_{i} | \alpha_{i}) =
(\alpha_{i}|\alpha_{i}) (\langle h_i, \lambda \rangle + k)$.

We need the following proposition for the proof of Theorem~\ref{thm:M}.

\begin{Prop} \label{prop:A} \hfill
\bna
\item If $\langle h_{i}, \lambda \rangle + k <0$, then
$\varphi_{k}=0$.
\item If $\langle h_{i}, \lambda \rangle + k \ge 0$, then
$\gamma \varphi_{k}$ is a monic polynomial of degree $\langle h_i,
\lambda \rangle +k$ in $t_{i}$.
\end{enumerate}
\end{Prop}

{}From now on, {\em a monic polynomial of degree $<0$ will be understood to be}
0. To prove Proposition \ref{prop:A}, we need some preparation.

\begin{Lem} \label{lem:P}
For $z \in K_{0}'= R(\beta) e(\beta-\alpha_{i}, i)
\otimes_{R(\beta-\alpha_i)} e(\beta-\alpha_i, i)
R^{\Lambda}(\beta)$, we have
\begin{equation} \label{eq:lemP}
F(z) x_{n+1} = F(z(x_{n} \otimes 1)) + E(z),
\end{equation}
where $E(z) \in R^{\Lambda}(\beta) \subset
 e(\beta,i)R(\beta+\alpha_i) e(\beta,i)\otimes_{R(\beta)} R^{\Lambda}(\beta)
 =E_{i}K_{0}$.
\end{Lem}
\begin{proof}
Write $z = a \otimes b$ where $a \in R(\beta) e(\beta-\alpha_i, i)$,
$b \in e(\beta-\alpha_i, i)R^{\Lambda}(\beta)$. Then
$$F(z) = a\tau_{n} b\quad\text{and}\quad E(z)=ab.$$ It follows that
\begin{equation*}
\begin{aligned}
F(z) x_{n+1} & = a \tau_{n} b  x_{n+1} = a \tau_{n} x_{n+1} b =
a(x_{n} \tau_{n} +1) b  = a x_{n} \tau_{n} b + ab \\
&= F(ax_{n} \otimes b) + E(a \otimes b) =F(z(x_{n} \otimes 1)) +
E(z),
\end{aligned}
\end{equation*}
as desired.
\end{proof}

By Corollary \ref{cor:eRe}, we have the $(R(\beta),
R^{\Lambda}(\beta))$-bimodule decomposition
\begin{equation} \label{eq:decomp}
\begin{aligned}
& e(\beta, i) R(\beta + \alpha_i) e(\beta, i) \otimes_{R(\beta)}
R^{\Lambda} (\beta) \\
& = F\bigl(R(\beta) e(\beta-\alpha_i, i) \otimes_{R(\beta -
\alpha_i)} e(\beta -\alpha_i, i) R^{\Lambda}(\beta)\bigr) \oplus
(R^{\Lambda}(\beta) \otimes \cor[t_{i}]) e(\beta, i),
\end{aligned}
\end{equation}
where $t_{i} = x_{n+1}$. Using the decomposition \eqref{eq:decomp},
we write
\begin{equation} \label{eq:psi}
P\bl e(\beta,i)\tau_{n} \cdots \tau_{1} x_{1}^{k} e(i,\beta)\br
= F(\psi_{k}) + \varphi_{k}
\end{equation}
for uniquely determined elements  $\psi_{k}\in K_0'$
and
$\varphi_{k}\in\cor[t_i]\otimes R^\Lambda(\beta)$.
Note that the definition of $\varphi_{k}$ coincides with the one
given in \eqref{eq:phi}. Indeed, for $k \ge 0$, we have
$$A(t_{i}^{k}) =AB( e(\beta,i)\tau_{n} \cdots \tau_{1} x_{1}^{k}
e(i,\beta))
=CP(e(\beta,i)\tau_{n} \cdots \tau_{1} x_{1}^{k} e(i,\beta))=C(\varphi_{k})=
\varphi_{k}.$$

Now we have
\begin{equation*}
\begin{aligned}
 F(\psi_{k+1}) + \varphi_{k+1} &= P(e(\beta,i)\tau_{n} \cdots \tau_{1}
x_{1}^{k+1}e(i,\beta)) = P(e(\beta,i)\tau_{n} \cdots \tau_{1} x_{1}^{k}e(i,\beta)) x_{n+1}\\
& = (F(\psi_{k}) + \varphi_{k}) x_{n+1} = F(\psi_{k}(x_{n} \otimes
1)) + E(\psi_{k}) + \varphi_{k} t_{i},
\end{aligned}
\end{equation*}
where $t_{i} = x_{n+1}$. Therefore we obtain
\begin{equation} \label{eq:recur}
\psi_{k+1} = \psi_{k}(x_{n} \otimes 1)\quad \text{and}\quad\varphi_{k+1} =
\varphi_{k} t_{i} + E(\psi_{k}).
\end{equation}
In particular, $\varphi_{k}$ is determined uniquely by $\varphi_{k+1}$.

\medskip

 Now we will prove Proposition \ref{prop:A}.
By Lemma~\ref{lem:ABP}, we have
$$g_{n} \cdots g_{1} x_{1}^{k}e(i,\nu)\x \tau_{1} \cdots
\tau_{n}= x_{n+1}^{k}a_i^\Lambda(x_{n+1}) \prod_{\nu_{a}
\neq i} Q_{i, \nu_{a}}(x_{n+1}, x_{a})e(\nu,i)$$ in $e(\beta,
i)R(\beta+\alpha_i)e(\beta, i) \otimes_{R(\beta)}
R^{\Lambda}(\beta)$, which implies
\begin{equation*}
\begin{aligned}
AB(g_{n} \cdots g_{1} x_{1}^k e(i,\nu))&= C \bigl(x_{n+1}^{k}a_i^\Lambda(x_{n+1})
\prod_{\nu_{a} \neq i} Q_{i, \nu_{a}}(x_{n+1}, x_{a}) e(\nu,i)\bigr)
\\&= t_{i}^ka_i^\Lambda(t_{i})\prod_{\nu_a \neq i} Q_{i,\nu_{a}}(t_{i}, x_{a})
e(\nu).
\end{aligned}
\end{equation*}
On the other hand, since $B$ is the map  taking the coefficient of
$\tau_n\cdots\tau_1$, we have
\begin{equation*}
\begin{aligned}
B\bl g_{n} \cdots g_{1} x_{1}^{k} e(i,\nu)\br& = B\Bigl(\prod_{\nu_{a}=i}
(-(x_{n+1}-x_{a})^2) x_{n+1}^k e(\nu,i)\tau_{n} \cdots \tau_{1}\Bigr) \\
& = t_{i}^k \prod_{\nu_{a}=i} (-(t_{i}-x_{a})^2) e(\nu).
\end{aligned}
\end{equation*}
Hence
\eq &&A\bigl(t_{i}^{k} \prod_{\nu_{a}=i} (t_{i} -
x_{a})^2 e(\nu)\bigr) = (-1)^{p} t_{i}^ka_i^\Lambda(t_{i})\prod_{\nu_{a}
\neq i} Q_{i,\nu_{a}}(t_{i}, x_{a}) e(\nu). \label{eq:AA} \eneq
Set \eqn
S&\seteq&\sum_{\nu\in I^\beta}\Bigl(\prod\limits_{\nu_{a}=i} (t_{i} -x_{a})^2
 e(\nu)\Bigr)
\in \cor[t_i]\otimes R^\Lambda(\beta), \\
F&\seteq&\gamma (-1)^{p}a_i^\Lambda(t_{i})\sum_{\nu\in I^\beta}
\Bigl(\prod_{\nu_{a} \neq i} Q_{i, \nu_{a}}(t_{i}, x_{a}) e(\nu)\Bigr)\in
\cor[t_i]\otimes R^\Lambda(\beta).\eneqn
Then they  are monic polynomials in $t_i$ of degree $2p$
and of degree
$\langle h_i,\Lambda - \beta \rangle + 2p = \langle h_i, \lambda \rangle + 2p$,
respectively. Note also that $S$ and $F$
belong to the center of $\cor[t_i]\otimes R^\Lambda(\beta)$.

Then \eqref{eq:AA} reads as
\eq
&&\text{$A(t_i^kS) =\gamma^{-1} t_i^kF$.}
\label{eq:SF}
\eneq
\begin{Lem} \label{lem:A}
We have \eq &&t_{i}^{k}F = (\gamma \varphi_{k}) S +
h_{k}, \label{eq:phik} \eneq where $h_{k}$ is a polynomial in
$t_{i}$ of degree $< 2p$. Hence $\gamma \varphi_{k}$ is the
quotient of $t_{i}^{k}F$ divided by $S$.
\end{Lem}
\begin{proof}
By \eqref{eq:recur}, we have $A(t_i^{k+1})-A(t_i^k)t_i\in
R^\Lambda(\beta)$, which implies \eq &&\text{$A(at_i)-A(a)t_i\in
R^\Lambda(\beta)$ for any $a\in R^\Lambda(\beta)[t_i]$.} \eneq By
induction on $m$. we shall show \eq &&\parbox{65ex}{for any
polynomial $f\in R^\Lambda(\beta)[t_i]$ in $t_i$ of degree $m$ and
$a\in R^\Lambda(\beta)[t_i]$, $A(af)-A(a)f$ is of degree $<m$.}
\label{eq:af} \eneq It is already proved for $m=0,1$. Hence it is
enough to show \eqref{eq:af} assuming that $f=t_ig$ and
\eqref{eq:af} is true for $g$. Then
$$A(af)-A(a)f=\bl A(at_ig)-A(at_i)g\br
+\bl A(at_i)-A(a)t_i\br g.
$$
The first term is of degree $<\deg(g)$
and the second term is of degree $\le\deg(g)$.
Hence we obtain \eqref{eq:af}.

\smallskip
Then we have
$$t_i^k\gamma^{-1}F-\varphi_kS=t_i^k\gamma^{-1}F-A(t_i^k)S=A(t_i^kS)-A(t_i^k)S$$
by \eqref{eq:SF}
and it is of order $<2p$ by applying
\eqref{eq:af} for $f=S$.
\end{proof}
As an immediate consequence of Lemma \ref{lem:A}, we see that
$\gamma \varphi_{k}$ is a monic polynomial in $t_{i}$ of degree
$\langle h_{i}, \lambda \rangle +k$.
This completes the proof of
Proposition \ref{prop:A}.

\bigskip

Now we are ready to finish the proof of Theorem \ref{thm:M}. By the
Snake Lemma, we have the following exact sequence of
$R^{\Lambda}(\beta)$-bimodules:
$$0 \longrightarrow \Ker A \longrightarrow q_{i}^{-2}
F_{i}^{\Lambda} E_{i}^{\Lambda} R^{\Lambda}(\beta) \longrightarrow
E_{i}^{\Lambda} F_{i}^{\Lambda} R^{\Lambda}(\beta) \longrightarrow
\Coker A \longrightarrow 0.$$

If $a\seteq\langle h_i, \lambda \rangle \ge 0$, by Proposition
\ref{prop:A}, we have
$$\Ker A =0, \quad \soplus_{k=0}^{a-1} \cor\, t_{i}^k \otimes
R^{\Lambda}(\beta) \overset{\sim} \longrightarrow \Coker A.$$ Hence
the composition
$\soplus_{k=0}^{a-1} \cor\, t_{i}^k \otimes
R^{\Lambda}(\beta) \To E_{i}^{\Lambda} F_{i}^{\Lambda} R^{\Lambda}(\beta)
\To\Coker A$ is an isomorphism, and
we obtain an isomorphism of $R^{\Lambda}(\beta)$-bimodules:
$$q_{i}^{-2} F_{i}^{\Lambda} E_{i}^{\Lambda} R^{\Lambda}(\beta)
\oplus \soplus_{k=0}^{a-1} \cor\, t_i^{k} \otimes R^{\Lambda} (\beta)
\overset{\sim} \longrightarrow E_{i}^{\Lambda} F_{i}^{\Lambda}
R^{\Lambda}(\beta),$$ which proves the statement (a)
in Theorem~\ref{thm:M}.

\medskip
Assume now $a\seteq -\langle h_{i}, \lambda \rangle \ge 0$. By Proposition
\ref{prop:A}, we have
$$\Coker(A) =0, \quad \Ker(A) = q^{2(\alpha_i| \Lambda -
\beta)} \bigoplus_{k=0}^{a-1} \cor\, t_{i}^k \otimes
R^{\Lambda}(\beta).$$
Then $\Ker A\to q_i^{-2}F_{i}^{\Lambda} E_{i}^{\Lambda} R^{\Lambda}(\beta)$ is given
by $t_i^k\longmapsto G(\psi_k)$.
We define a map $$\Psi\cl \Ker(A)
\longrightarrow \bigoplus_{k=0}^{a-1} \cor\, t_{i}^k\otimes
R^{\Lambda}(\beta)$$
as the composition
$$\Ker A
\longrightarrow q_i^{-2} F_{i} E_{i} R^{\Lambda}(\beta) \to [\;E
\circ(x_{n}^k\otimes 1)\;]\bigoplus_{k=0}^{a-1} \cor\,
t_i^{k}\otimes R^{\Lambda} (\beta),$$ where the map $\Ker A \to
q_i^{-2}F_{i} E_{i} R^{\Lambda} (\beta)$ is given by $t_{i}^{k}
\mapsto \psi_{k}$ and $E \circ(x_{n}^k\otimes 1)\cl q_i^{-2} F_{i}
E_{i} R^{\Lambda}(\beta) \to\bigoplus_{k=0}^{a-1} \cor\,
t_i^{k}\otimes R^{\Lambda} (\beta)$ is given by $s\longmapsto
\sum_{k=0}^{a-1} t_i^k\otimes E(s(x_n^k\otimes 1))$. Then we have a
commutative diagram
$$\xymatrix{
\Ker(A)\ar[r]\ar[d]&q_i^{-2} F_{i} E_{i} R^{\Lambda}(\beta)
\ar[d]^{E \circ(x_{n}^k\otimes 1)}\ar[dl]_-{G}\\
q_i^{-2}F^\Lambda_{i} E^\Lambda_{i} R^{\Lambda}(\beta)\ar[r]
&\soplus_{k=0}^{a-1} \cor\, t_i^{k}\otimes R^{\Lambda} (\beta).
}$$
Then by \eqref{eq:recur}, we have
$$\Psi(t_{i}^{j}) = \sum_{k=0}^{a-1} E(\psi_{j}(x_{n}^k \otimes 1))t_i^k =
\sum_{k=0}^{a-1} E(\psi_{j+k}) t_{i}^k.$$ Since $\gamma \varphi_{k}$
is a monic polynomial of degree $k-a$ in $t_{i}$, we deduce
$$ \varphi_{k} = \begin{cases} 0 \ \ & \text{if} \ k<a, \\
\gamma^{-1} \ \ & \text{if} \ k=a. \end{cases}$$
{}From the relation
$$\gamma^{-1} = \varphi_{a} = E(\psi_{a-1}) + \varphi_{a-1} t_{i},$$
we obtain $$E(\psi_{a-1})=\gamma^{-1}.$$ For $k \le a-1$, we have
$$0 = \varphi_{k} = \varphi_{k-1} t_{i} + E(\psi_{k-1}),$$
from which we obtain
$$E(\psi_{k})=0 \ \ \text{for all} \ k < a-1.$$
Thus we derive a system of equations
$$\Psi(t_{i}^{j}) =\gamma^{-1} t_{i}^{a-1-j} + \sum_{a-1-j<k\le a-1} E(\psi_{j+k})
t_{i}^{k}.$$
Hence $\Psi$ is an $R^\Lambda(\beta)$-linear endomorphism of
$\soplus_{k=0}^{a-1} \cor\, t_{i}^k \otimes R^{\Lambda}(\beta)$
which is in a triangular form. Therefore, $\Psi$ is an
isomorphism and we conclude
$$F_{i}^{\Lambda} E_{i}^{\Lambda} R^{\Lambda}(\beta) \overset{\sim}
\longrightarrow E_{i}^{\Lambda} F_{i}^{\Lambda} R^{\Lambda}(\beta)
\oplus \bigoplus_{k=0}^{a-1} \cor\, t_{i}^{k} \otimes
R^{\Lambda}(\beta) $$ as $R^{\Lambda}(\beta)$-bimodules.
This completes the proof of Theorem~\ref{thm:M}. \qed

\vskip 3em


\section{Categorification of $V(\Lambda)$}

In this section, we shall show that cyclotomic \KLRs\ categorify the
irreducible highest weight module $V(\Lambda)$. In \cite{R08}, one
can find a more systematic and detailed treatment of the
categorification.

Hereafter, we assume that the degree-zero part $\cor_0$ of the base
ring $\cor$ is a commutative field. For $\beta\in Q^+$, let us
denote by $\Proj(R^\Lambda(\beta))$ the category of finitely
generated projective graded $R^\Lambda(\beta)$-modules, and by
$\Rep(R^\Lambda(\beta))$ the category of graded
$R^\Lambda(\beta)$-modules that are finite-dimensional over
$\cor_0$. Let us denote by $[\Proj(R^\Lambda(\beta))]$ and
$[\Rep(R^\Lambda(\beta))]$ their Grothendieck groups. Then they are
$\Z[q,q^{-1}]$-modules, where the action of $q$ is given by the
grade shift functor $q$ (see \eqref{eq:shift}). Let us set
$$[\Proj(R^\Lambda)]\seteq\soplus_{\beta\in Q^+}[\Proj(R^\Lambda(\beta))]
\quad\text{and}\quad
[\Rep(R^\Lambda)]\seteq\soplus_{\beta\in Q^+}[\Rep(R^\Lambda(\beta))].$$

By Theorem ~\ref{th:proj} and its corollary
(Corollary~\ref{cor:exact}),
the arrows of the following diagrams are exact functors: \eqn &&
\xymatrix@C=13ex{\Proj(R^\Lambda(\beta))\ar@<.8ex>[r]^-{F_{i}^{\Lambda}}
&\Proj(R^\Lambda(\beta+\alpha_i)) \ar@<.8ex>[l]^-{q_i^{1-\lan
h_i,\Lambda-\beta\ran} E_{i}^{\Lambda}}
},\\
&&\xymatrix@C=13ex{\Rep(R^\Lambda(\beta))
\ar@<.8ex>[r]^-{q_i^{1-\lan h_i,\Lambda-\beta\ran} F_{i}^{\Lambda}}
&\Rep(R^\lambda(\beta+\alpha_i))\ar@<.8ex>[l]^-{E_{i}^{\Lambda}}. }
\eneqn Hence they induce endomorphisms $\F$ and $\E$ on
$[\Proj(R^\Lambda)]$ and $[\Rep(R^\Lambda)]$. The following lemma
immediately follows from Theorem~\ref{thm:M}.

\begin{Lem} \label{lem:mixed}
For all $i,j \in I$, we have
$$[\E, \F[j]] = \delta_{ij} \dfrac{K_{i} -
K_{i}^{-1}}{q_i - q_i^{-1}} $$ on $[\Proj(R^\Lambda)]$ and
$[\Rep(R^\Lambda)]$. Here, $K_i$ is given by
$$K_{i}\vert_{[\Proj(R^\Lambda(\beta))]}=q_i^{\lan
h_i,\Lambda-\beta\ran}, \quad
K_{i}\vert_{[\Rep(R^\Lambda(\beta))]}=q_i^{\lan
h_i,\Lambda-\beta\ran}.$$
\end{Lem}

It is obvious that the action of $\E$ on
 $[\Proj(R^\Lambda)]$ and $[\Rep(R^\Lambda)]$ are locally nilpotent.
Lemma~\ref{lem:integrable} implies that
the action of $\F$ on
 $[\Proj(R^\Lambda)]$ and $[\Rep(R^\Lambda)]$ are also locally nilpotent.
Therefore, by \cite[Proposition B.1]{KMPY96}, the Grothendieck groups
$$[\Proj(R^{\Lambda})]_{\Q(q)} = \Q(q) \otimes_{\Z[q,q^{-1}]}
[\Proj(R^{\Lambda})]$$ and
$$[\Rep(R^{\Lambda})]_{\Q(q)} = \Q(q) \otimes_{\Z[q,q^{-1}]}
[\Rep(R^{\Lambda})]
$$
become integrable $U_q(\g)$-modules.

\nc{\inn}{\mathbin{\rule{.4pt}{7pt}\kern-4pt\cup}}
For a left $R^\Lambda$-module $N$, let us denote by
$N^\psi$ the right $R^\Lambda$-module
obtained from $N$ by the anti-involution $\psi$ of $R^\Lambda$
that fixes the generators $x_k$, $\tau_l$ and $e(\nu)$.
By the pairing
\eqn
[\Proj(R^\Lambda)]\times [\Rep(R^{\Lambda})]&\To&\Z[q,q^{-1}],\\
\hs{5ex}\inn\hs{10ex}&&\hs{4ex}\inn\\
\hs{5ex}(P,M)\hs{7ex}&\longmapsto&
\sum_{n\in\Z} q^n\dim_{\cor_0}(P^\psi\otimes_{R^\Lambda}M)_n
\eneqn
the free $\Z[q,q^{-1}]$-modules
$[\Proj(R^\Lambda)]$ and $[\Rep(R^{\Lambda})]$ are dual to each other.
Moreover, $\E\vert_{[\Proj(R^\Lambda)]}$ and
$\F\vert_{[\Proj(R^\Lambda)]}$ are adjoint to
$\F\vert_{[\Rep(R^{\Lambda})]}$ and $\E\vert_{[\Rep(R^{\Lambda})]}$,
respectively.

We denote by $\Rep(R(\beta))$ the category of $R(\beta)$-modules
that are finite-dimensional over $\cor_0$. We define
$[\Rep(R)]=\soplus_{\beta\in Q^+}[\Rep(R(\beta))]$,
$[\Rep(R)]_{\Q(q)}=\Q(q)\otimes_{\Z[q,q^{-1}]}[\Rep(R)]$,
$\Proj(R(\beta))$, $[\Proj(R)]$ and $[\Proj(R)]_{\Q(q)}$, similarly.
Then $[\Proj(R)]$ and $[\Rep(R)]$ are also dual to each other. The
fully faithful exact functor
$\Rep(R^\Lambda(\beta))\to\Rep(R(\beta))$ induces a
$\Z[q,q^{-1}]$-linear homomorphism $[\Rep(R^\Lambda)]\to
{[\Rep(R)]}$. It is well-known that $[\Rep(R)]$ (resp.\
$[\Rep(R^\Lambda)]$) has a basis $[S]$ where $S$ ranges over the set
of the isomorphism classes of irreducible $R$-modules (resp.\
$R^\Lambda$-modules). Hence  $[\Rep(R^\Lambda)]\to {[\Rep(R)]}$ is
injective and its cokernel is a free $\Z[q,q^{-1}]$-module. By the
duality, the homomorphism $[\Proj(R)]\to{[\Proj(R^\Lambda)]}$
(induced by the functor $R^\Lambda(\beta)\otimes_{R(\beta)}\scbul$)
is surjective. Note that $[\Proj(R)]\to{[\Proj(R^\Lambda)]}$ is
$U^-_\A(\g)$-linear. In \cite{KL09}, Khovanov and Lauda showed that
$[\Proj(R)]$ is isomorphic to $U^-_\A(\g)$ as bialgebras.
Hence its quotient $[\Proj(R^\Lambda)]$ is generated by the trivial
representation $\mathbf{1}_{\Lambda}$ of $R^{\Lambda}(0)$. Therefore
 $[\Proj(R^\Lambda)]_{\Q(q)}$ is an integrable highest weight $U_q(\g)$-module
and it is isomorphic to $V(\Lambda)$ by Proposition \ref{prop:hw}
(a). Hence
 $[\Proj(R^\Lambda)]$ is isomorphic to $V_\A(\Lambda)$.
By duality, we obtain $[\Rep(R^{\Lambda})] \simeq V_{\A}(\Lambda)^{\vee}$.

\medskip
To summarize, we obtain the categorification of the irreducible
highest weight module $V(\Lambda)$.

\begin{Thm} \label{thm:N}
There exist isomorphisms of $U_\A(\g)$-modules
$$ \text{$[\Proj(R^{\Lambda})] \simeq V_{\A}(\Lambda)$
 and $[\Rep(R^{\Lambda})] \simeq V_{\A}(\Lambda)^{\vee}$.}$$
\end{Thm}

\vskip 3em


\bibliographystyle{amsplain}


\end{document}